\documentclass[12pt]{article}

\usepackage{fullpage}
\usepackage[title,titletoc]{appendix}
\usepackage[authoryear,round]{natbib}
\usepackage{caption}
\usepackage[hang,flushmargin]{footmisc}
\usepackage[all]{xy}
\usepackage{rotating}
\usepackage{amsmath}
\usepackage{amssymb}
\usepackage{amsthm}
\usepackage{amscd}
\usepackage{amsfonts}

\usepackage{color}
\usepackage{xcolor}
\usepackage{amsmath,bm}
\usepackage{enumerate}
\usepackage{booktabs}
\usepackage{amssymb}

\usepackage{bbm}

\DeclareMathOperator{\sign}{sign}

\newcommand{\rr}{\mathbbm{R}}

\newcommand{\NN}{\mathcal{N}}
\newcommand{\zbar}{\overline{z}}
\def\ignore#1{}
\newcommand{\aipp}{a preserver}
\newtheorem{proposition}{Proposition}
\newtheorem{theorem}{Theorem}
\newtheorem{lemma}{Lemma}
\newtheorem{remark}{Remark}
\newtheorem{definition}{Definition}

\date{}
\begin{document}

\title {Continuous approximations for the fixation probability of the Moran processes on star graphs}

\author{Poly H. da Silva\thanks{Department of Statistics, Columbia University, USA, email: polyhdasilva@gmail.com}, Max O. Souza\thanks{\raggedright Instituto de Matemática e Estatística, Universidade Federal Fluminense, Brazil, \mbox{email: maxsouza@id.uff.br}}}

\maketitle

\abstract{We consider a generalized version of the birth-death (BD) and death-birth (DB) processes introduced by Kaveh, Komarova, and Kohandel (2015), in which two constant fitnesses, one for birth and the other for death, describe the selection mechanism of the population. Rather than constant fitnesses, in this paper we consider more general 
frequency-dependent fitness functions (allowing any smooth functions) under the weak-selection regime. A particular case arises in evolutionary games on graphs, where the fitness functions are linear combinations of the frequencies of types. For a large population structured as a star graph,  we provide approximations for the fixation probability which are solutions of certain ODEs (or systems of ODEs).  For the DB case, we prove that our approximation has an error of order $1/N$, where $N$ is the size of the population.

The general BD and DB processes contain, as special cases, the BD-* and DB-* (where * can be either B or D) processes described in Hadjichrysanthou, Broom, and Rycht{\'a}{\v{r}} (2011) --- this class includes many examples of update rules used in the literature. Our analysis shows how the star graph may act as an amplifier, suppressor, or remains isothermal depending on the scaling of the initial mutant placement. We identify an analytical threshold for this transition and illustrate it through applications to evolutionary games, which further highlight asymmetric structural effects across different game types. Numerical examples show that our fixation probability approximations remain accurate even for moderate population sizes and across a wide range of frequency-dependent fitness functions, extending well beyond previously studied linear cases derived from evolutionary games, or constant fitness scenarios.}

\vspace{0.5em}

\maketitle

\section{Introduction}\label{sec1}
\subsection{Background}
\label{ssec:back}

The use of stochastic processes to understand the evolutionary dynamics goes back at least to Galton \citep{watson1875probability}, who devised a process to model the extinction of aristocratic family names, and  that today bears his name. Early in the twentieth century, Wright and Fisher introduced a stochastic process that was a watershed in the study of mathematical population genetics --- now known as the Wright-Fisher (WF) process \citep{Wright1,Wright2}. 

 Later on, in the early sixties, Moran devised a simplified process as alternative to the WF process: a birth-death process now known as the Moran process \citep{moranp}. The Moran process considers a finite well-mixed population where each individual can interact with every other individual, with two types (or traits), and such that mutations are not considered. 
 As a result, two homogeneous states turn out to be  absorbing, i.e., the dynamics eventually reaches one of them --- when this happens we say that the corresponding type has fixed. This can also be considered as a special case of the Kimura class of processes studied in \cite{chalubsouza}, and for processes in this class an important issue  is the computation of the fixation probability,  
 i.e. the probability that a given type will fix conditional on the current state of the population. 
 
Although the original Moran process only considers constant fitnesses, the more recent versions involve frequency dependent fitnesses --- cf. \cite{nowak1,ChalubSouza09b,AssafMobilia2010,MobiliaAssaf2010,NowakSasakiTaylorFudenberg,TaylorFudenbergSasakiNowak,Traulsenetal2006b}. 
The classical Moran process and other models that assume well-mixed populations can be seen as an interacting population dynamics on complete graphs, in which any pair of individuals can be in interaction with each other. However when the interactions are restricted to certain pairs of individuals, this can be generalized to the spatial model, where the interaction can only occur between two neighbours in a given graph.
 
 The use of well-mixed populations has been studied as early as the results in \cite{erdos1960evolution,nagylaki1980numerical,barabasi1999emergence}. In 2005,  \cite{nowak1} brought the study of population dynamics on graphs to the mainstream of evolutionary game theory.  
 
 The computation of the fixation probability in this framework is more involved though --- \cite{nowak1} identified the class of isothermal graphs, which are a subset of regular graphs for which  the fixation probability can be calculated in a similar way to the complete graph. Due to this difficulty, \cite{nowak1} studied the fixation probability when there is only one mutant in a star graph (invasion probability), and found an approximation when the fitness is constant and the number of leaves is large.  \cite{br2008} found the exact formula for the invasion probability in a star graph and its asymptotics revisiting the asymptotic results in \cite{nowak1}. The latter computation was amended in \cite{chalub2014asymptotic}. The exact formula for the fixation probability for any initial state was later given in \cite{monk2014martingales}.

There are two versions of the Moran process on graphs, differing by the order in which birth and death events occur.  In the Birth-Death (BD) process, an individual is first chosen with a probability proportional to its birth fitness to reproduce. Subsequently, an individual among its neighbors is chosen with a probability proportional to its death fitness (death propensity) to die, and then be replaced by the new offspring. On the other hand, in the Death-Birth (DB) process, an individual is first chosen with a probability proportional to its death fitness to die, then an individual among its neighbors is selected with a probability proportional to its birth fitness to reproduce, and its offspring replace the individual that died. For different updating mechanism versions,  see \cite{hcbr}.

Population structure can introduce diverse phenomena in evolution. Comparing an evolutionary process on a given graph (and its invasion probabilities) with the corresponding process on the complete graph under a constant fitness function, \cite{nowak1} introduced the notions of accelerator and suppressor graphs. In this context, \citeauthor{nowak1} proved that the star graph with a BD update mechanism is an accelerator of evolution, meaning that it amplifies the effects of natural selection. On the other hand, \cite{kkk} and \cite{hindersin2015most} showed that DB processes on star graphs act as suppressors of evolution.
Similarly, \cite{pattni2015evolutionary} studied several classes of graphs and updating rules that have the same probability of fixation as the standard Moran process in a complete graph, when fitness is constant. In particular, these processes have the same invasion probability for each $N$.

Existing literature predominantly focuses on models for which birth and death fitnesses are constant or linear functions, as in classical evolutionary game theory. 
However, in this paper,  instead of constant fitness values as seen in \cite{kkk} and \cite{pattni2015evolutionary}, or  linear functions, we consider much more general frequency-dependent fitness functions for both birth and death. These functions can be any smooth functions of the frequency of types. For a large population,  our study provides continuous approximations for the fixation probability in star graphs with general birth and death fitnesses under the weak-selection regime. For the DB process, we prove that our continuous approximation for the fixation probability is of order $1/N$, where $N$ is the size of the population. It is worth noting that although continuous approximations are provided, we do not consider an infinite population limit here. In this vein,
our approach is similar to the one in   \cite{chalubsouza16}, however, our techniques diverge significantly.

We also investigate how population structure affects the invasion dynamics of rare mutants, focusing on the behavior of the ratio between fixation probabilities in the star graph and in the complete graph. Our analysis shows that the star can act as an amplifier, suppressor, or behave isothermally, depending on how the initial mutant is distributed—whether placed deterministically at the center, uniformly at random, or according to a vanishing probability scaling. For the DB process, we derive an explicit analytical threshold that separates suppressor and amplifier regimes, and we demonstrate how this threshold depends on the selection intensity and the choice of fitness functions. These findings highlight a subtle interplay between selection, structure, and initial conditions that had not been fully captured in prior work.

In addition, we present a series of numerical examples that confirm the accuracy and flexibility of our continuous approximations for fixation probabilities, even for moderate population sizes. These examples showcase the broader applicability of our approach beyond previously studied linear or constant-fitness settings. Together, our results contribute to a more complete understanding of evolutionary dynamics in structured populations under general fitness functions.

\subsection{Outline}
The paper is organized as follows. Section~\ref{sec:EDG}  presents the preliminaries and formulates the problem. In Section~\ref{sec:odecandidate}, we derive continuous approximations for the fixation probability for both the BD and DB processes, considering general functions for birth and death fitnesses under weak-selection regime. For the DB process, we establish in Theorem~\ref{rDB} that the approximation error is of order $1/N$, with the proof provided in Appendix~\ref{secA1}. For the BD process, we obtain an approximation closely resembling the BD process on a complete graph, as shown in~\cite{chalubsouza16}. In Section~\ref{sec:2player}, we examine the asymptotic qualitative behavior of a structured population as a star graph when fitness is a linear function defined by a payoff matrix. In Section~\ref{sec:IP}, we analyze the invasion probability—the likelihood that a single mutant eventually fixes in the population. Finally, Section~\ref{sec:NE} presents numerical examples illustrating that the approximations derived in Section~\ref{sec:odecandidate} for both DB and BD processes closely match the fixation probability.

\subsection{On the notations used}
In the following, we write $\bm{M}=[m_{ij}]$ to denote  a matrix in $\rr^{n\times n}$, and $\bm{M}^t$ for the transpose of $\bm{M}$. We also write $\vert\cdot\vert$ for the vector infinity norm in $\rr^n$ and $\Vert \cdot \Vert$ to denote the corresponding  matrix norm,  
which is  given by $\Vert \bm{M} \Vert=\max_{1\leq i\leq n}\sum_{j=1}^{n} \vert m_{ij}\vert$, i.e., the maximum absolute row sum of the matrix. It will be convenient to regard a vector $\bm{v}$ as a matrix in $\rr^{n\times 1}$, whose entries will be written as  $v_{i1}:=v_i$. Also, we denote by $\bm{I}_n$ the $n\times n$ identity matrix. We drop the subscript when it is clear.

\section{Evolutionary dynamics on graphs}\label{sec:EDG}

\subsection{Preliminaries}

Consider a finite population of $N$ individuals divided into two types (traits), $A$, the wild-type or resident, and $B$, the mutant. Assume that each sub-population is homogeneous, i.e. there is no advantage of any particular individual with respect to others with the same type. The number of individuals is always assumed to be constant. The population is structured, which means that each individual may interact, in different ways, with other individuals in the population according to the geometric structure of a graph that represents the interaction pattern of a given population.

More precisely, let $G=(V, E)$ be a finite, simple (without loops and parallel edges), undirected and connected graph, where $V$ is the set of vertices and $E$ is the set of edges. One unique individual lives at each vertex of the graph $G$, and the vertices $i$ and $j$ are connected if an interaction is possible between the individuals living at them. In other words, each individual only interacts with its neighbours. A special case of this, namely the classical Moran process, is the model for which interactions are allowed for every pair of individuals, i.e., the underlying graph is complete. For simplicity, we call an individual sited at the vertex $i$ the  individual $i$. 

Let $\mathbb R_+$ be the set of all strictly positive real numbers, and $C^\infty([0,1], \mathbb R_+)$ be the space of all smooth (infinitely times differentiable) functions from the closed interval $[0,1]$ to $\mathbb R_+$. Let $\varphi_1^A,\varphi_2^A:[0,1]\longrightarrow \rr_+$ be  frequency dependent fitness functions ($\varphi_1^A, \varphi_2^A\in C^\infty([0,1],\mathbb R_+)$) such that $\varphi_1^A(x)$ and $\varphi_2^A(x)$  represent the birth fitness and the death fitness (or death propensity) of type $A$, respectively,  when there exist $(1-x)N$ individuals of type $A$ in the population for $x\in\{0,1/N,2/N,...,1\}$. In order to balance selection and stochastic drift when the population is large, we assume that we are in the weak-selection regime. In other words, let  $\varphi_1^B,\varphi_2^B:[0,1]\longrightarrow \rr_+$ be functions in $C^\infty([0,1],\mathbb R_+)$, and for $i=1,2$, define $\rho_i:[0,1]\longrightarrow \rr$ with 
\[\varphi_i^B(x)=\varphi_i^A(x)+\frac{\rho_i(x)}{N}.\] This implies in $ \rho_i\in C^\infty([0,1],\mathbb R)$, for $i=1,2$. Note that $\varphi_1^B(x)$ and $\varphi_2^B(x)$ represent the birth fitness and the death fitness of type $B$, respectively, when there exist $xN$ individuals of type $B$ in the population for $x\in\{0,1/N,2/N,...,1\}$. In what follows, it will be convenient as a simplification device to define $\varphi_i:[0,1]\longrightarrow \rr_+$ by 
\[\varphi_i(x):=\frac{\varphi_i^B(x)}{\varphi_i^A(x)}=1+\frac{\psi_i(x)}{N},\] 
where $\psi_i(x)=\rho_i(x)/\varphi_i^A(x)$, for $i=1,2$. We assume that $\varphi_i^A,\varphi_i^B$ ($i=1,2$) are chosen such that $\psi_1, \psi_2\in C^\infty([0,1],\mathbb R)$. Then, in practice, we consider that the fitness functions for $B$ are $\varphi_1$ and $\varphi_2$, and the fitness functions for $A$ are both equal to $1$. 

\subsection{Birth-death and Death-birth processes on graphs}

A simple stochastic  process is intuitively taken as one where each population update consists of two events: the birth of a single individual and the death of a single individual. The order of these events does matter and the difference may be quite significant in non-complete graphs. We follow \cite{kkk}, and consider a BD and DB processes allowing for selection both on birth and death. This formulation contains as special cases the BD-* and DB-* (where * can be either B or D) processes described in \cite{hcbr}, which in turn include many examples of update rules used in the literature. Unlike the constant rates for birth and death as in  \cite{kkk},  this paper considers general frequency-dependent birth and death fitness functions in $C^\infty([0,1], \mathbb R_+)$.

A general BD process in a structured population is defined as the following Markov chain. At each step, an individual $i$ is selected from the population to reproduce with probability proportional to its birth fitness, and an individual among the neighbours of $i$, say $j$, is selected to die with probability proportional to its death fitness among all neighbours of $i$. Upon a selection event for $i$ and $j$, the individual $j$ dies and is replaced by an offspring of individual $i$. Given a population of $N$ individuals  structured as a graph $G=(V,E)$,  with $k$ mutants living on a subset $K\subset V$, the probability that a new mutant appears -- that is a resident individual living on $V\setminus K$ dies and be replaced by an offspring of a mutant living on $K$ in the BD mechanism --  is given by

\[\mathbbm{P}_K^{BD}(k\rightarrow k+1)=\frac{\varphi_1(\frac{k}{N})}{k\varphi_1(\frac{k}{N})+N-k}\sum\limits_{i\in K}\frac{d_i-n_i}{n_i\varphi_2(\frac{k}{N})+d_i-n_i},\]

\noindent where 
$d_i$ is the degree of vertex $i$, and $n_i=n_i(K)$ is the number of mutant neighbours of vertex $i$. Likewise, the probability that a new wild-type appears in the BD process is

\[\mathbbm{P}_K^{BD}(k\rightarrow k-1)=\frac{1}{k\varphi_1(\frac{k}{N})+N-k}\sum\limits_{i\in V\setminus K}\frac{n_i\varphi_2(\frac{k}{N})}{n_i\varphi_2(\frac{k}{N})+d_i-n_i}.\]

The mechanism of a general DB process in a structured population is the opposite of the BD process. At each step, first an individual $i$ is selected from the population to die with probability proportional to its death fitness, and then an individual among the neighbours of $i$, say $j$, is selected, to reproduce,  with probability proportional to its birth fitness among all neighbours of $i$. Upon a selective event, individual $i$ dies and is replaced by an offspring of individual $j$.  Given a population of $N$ individuals structured as a graph $G$ with $k$ mutants, the probability that a new mutant appears  in the DB process is

\[\mathbbm{P}_K^{DB}(k\rightarrow k+1)=\frac{1}{k\varphi_2(\frac{k}{N})+N-k}\sum\limits_{i\in V\setminus K}\frac{n_i\varphi_1(\frac{k}{N})}{n_i\varphi_1(\frac{k}{N})+d_i-n_i},\]

\noindent and the probability that  in the DB process a new wild-type appears is

\[\mathbbm{P}_K^{DB}(k\rightarrow k-1)=\frac{\varphi_2(\frac{k}{N})}{k\varphi_2(\frac{k}{N})+N-k}\sum\limits_{i\in K}\frac{d_i-n_i}{n_i\varphi_1(\frac{k}{N})+d_i-n_i}.\]

As we mentioned before, in the general BD and DB processes, the two homogeneous states are absorbing,  so the dynamics eventually reaches one of them and when this happens we say that the corresponding type has fixed. The fixation probability is the probability that a given type will fix given the current state of the population.  
In this paper, we 
develop a method to estimate the fixation probability of the BD and DB processes for populations on finite star graphs with general
birth and death fitnesses under the weak-selection regime.

A star graph with $N$ vertices is a connected undirected simple graph that has only one vertex of degree $N-1$, called the center, while all the other vertices, leaves, have degree one. 
Since the permutations on leaves give isomorphic graphs, the dynamics of the process can be described by the number of mutants at the leaves and the type of the individual living at the center. In the BD process, we denote by $p^{BD,N}_{1,x}$ (resp. $p^{DB,N}_{1,x}$ in the DB process)  the fixation probability of type $B$ when the initial state of the process is the star graph with a mutant, $B$, living at its center and $xN$ mutants living at the leaves. Similarly, denote by $p^{BD,N}_{2,x}$ (resp. $p^{DB,N}_{2,x}$ in the DB process) the fixation probability of $B$ when the initial state is the star graph with a wild-type, $A$, at its center and $xN$ mutants living at its leaves. In the sequel, by removing the superscripts ``BD'' and ``DB'' in a statement, an equation, etc., we mean it is true for both BD and DB processes. From the one-step analysis of the fixation probabilities, we can obtain the following $2N$ recursive equations

\begin{equation}\label{p}
\begin{array}{c}
p^{N}_{1,x}= a_x(\overline{z}) p^{N}_{2,x}+b_x(\overline{z}) p^{N}_{1,x+\zbar}\\\\
p^{N}_{2,x}=c_x(\overline{z}) p^{N}_{2,x-\zbar}+d_x(\overline{z}) p^{N}_{1,x}
\end{array}
\end{equation}

\noindent for $x\in\{0,\zbar,2\zbar,...,1-\zbar\}$ and $\overline{z}=1/N$, with boundary conditions $p^{N}_{2,0}=0$ and $p^{N}_{1,1- \zbar}=1$, where $a_x, b_x, c_x$ and $d_x$ are continuous functions on $z\in[0,\delta]$, $\delta >\zbar$ defined for the BD process as
\begin{equation}\label{coefBD}
\begin{array}{ll}
{\displaystyle a^{BD}_x(z):=}&{\displaystyle \frac{1-z+xz\psi_2(x+z)}{1+xz\psi_2(x+z)+z^2\psi_1(x+z)}},\\\\
{\displaystyle b^{BD}_x(z):=}&{\displaystyle \frac{z+z^2\psi_1(x+z)}{1+xz\psi_2(x+z)+z^2\psi_1(x+z)}},\\\\
{\displaystyle c^{BD}_x(z):=}&{\displaystyle\frac{z(1+z\psi_2(x))}{(1+z\psi_1(x))(1-z+xz\psi_2(x))+z(1+z\psi_2(x))}},\\\\
{\displaystyle d^{BD}_x(z):=}&{\displaystyle\frac{(1+z\psi_1(x))(1-z+xz\psi_2(x))}{(1+z\psi_1(x))(1-z+xz\psi_2(x))+z(1+z\psi_2(x))}},
\end{array}
\end{equation} 
\noindent and for DB process as
\begin{equation} \label{coefDB}
\begin{array}{ll}
{\displaystyle a^{DB}_x(z):=}&{\displaystyle\frac{z(1+z\psi_2(x+z))}{1+xz\psi_1(x+z)+z^2\psi_2(x+z))}},\\\\
{\displaystyle b^{DB}_x(z):=}&{\displaystyle\frac{1-z+xz\psi_1(x+z)}{1+xz\psi_1(x+z)+z^2\psi_2(x+z))}},\\\\
{\displaystyle c^{DB}_x(z):=}&{\displaystyle\frac{(1+z\psi_2(x))(1-z+xz\psi_1(x))}{(1+z\psi_2(x))(1-z+xz\psi_1(x))+z(1+z\psi_1(x))}},  \\\\
{\displaystyle d^{DB}_x(z):=}&{\displaystyle\frac{z(1+z\psi_1(x))}{(1+z\psi_2(x))(1-z+xz\psi_1(x))+z(1+z\psi_1(x))}.}
\end{array}
\end{equation}

 Note that in the original BD and DB processes, there is a positive probability of not jumping at each step, i.e.  the individuals selected for birth and death may be of the same type, resulting in no change. However if the process is not in the absorbing states, then with probability one, it eventually jumps to another state in finite time. So by considering transition probabilities conditioned on jumping to another state, we can construct a new Markov chain that leaves the current state with probability one at each step. More precisely, the transition probability of jumping to a neighbour for the new reduced Markov chain is the conditional probability of jumping to that neighbour provided that a jump occurs. It is clear that the fixation probabilities of mutants for the original BD and DB processes are the same as those in the reduced Markov chains. We denote by $\bm{L}$ the conditional transition probability matrix on the star graph with dimension ${2N\times2N}$. For each fixed $N$, $\bm{L}$ is defined as follows
\[L_{ij}=\left\lbrace \begin{array}{ll}
1, & \text{ if }   i=j=1 \text{ or } i=j=2N; \\
c_{(i-1)\zbar}(\zbar), & \text{ if } j= i-1 \text{ and } 2\leq i\leq N; \\
d_{(i-1)\zbar}(\zbar), & \text{ if } j= N+i \text{ and } 2\leq i\leq N; \\
a_{(i-N-1)\zbar}(\zbar), & \text{ if } j=i-N \text{ and } N+1\leq i\leq 2N-1; \\
b_{(i-N-1)\zbar}(\zbar), & \text{ if } j=i+1 \text{ and } N+1\leq i\leq 2N-1; \\
0, &\text{otherwise.}
\end{array}\right.\]
 \noindent where $a_x$, $b_x$, $c_x$ and $d_x$ are defined in (\ref{coefBD}) for the BD process and in (\ref{coefDB}) for the DB process. Also, we define $\bm{M}:= \bm{L} -\bm{ I}$ and denote by \[\bm{F}=\left[ \begin{array}{c}
 p^{N}_{2,0}\\
 p^{N}_{2,\zbar}\\
\vdots\\
 p^{N}_{2,1-\zbar}\\
 p^{N}_{1,0}\\
 p^{N}_{1,\zbar}\\
\vdots\\
 p^{N}_{1,1-\zbar}
\end{array}\right]   \] 
the fixation probability vector. Note that, $\bm{L F} =\bm{F} $ and so $\bm{M  F} =\bm{0}$.

We are now ready to find continuous approximation candidates for the fixation probabilities for both the BD and DB processes on a star graph under the weak-selection regime. In the next section, we see that although the method to derive our continuous approximation of the fixation probability is quite similar for both the BD and DB processes, the resulting approximations have distinct forms. 
For the DB process, we establish in Theorem~\ref{rDB} that the error of the approximation is of order $1/N$. As for the BD process, we find an approximation very similar to the BD process on a complete graph, as given in~\cite{chalubsouza16}.

\section{Finding continuous approximation candidates}\label{sec:odecandidate}

In order to find continuous approximation candidates for the fixation probability for the BD and DB processes on a star graph,  we first suppose that there exist smooth functions $q_1,q_2:[0,1]\times [0,\delta]\longrightarrow [0,1]$ such that
\begin{equation}\label{recureq}
\left\lbrace \begin{array}{c}
a_x(z) q_2(x,z)+b_x(z) q_1(x+z,z)-q_1(x,z)=0\\\\
c_x(z) q_2(x-z,z)+d_x(z) q_1(x,z)-q_2(x,z)=0
\end{array}\right. 
\end{equation} 

Since $a_x, b_x,c_x$ and $d_x$ are smooth functions, for a sufficiently small $z$,  we can use the Taylor series at point $(x,0)$ and rewrite each equation in (\ref{recureq}) as a Taylor polynomial.
As the equations in (\ref{recureq}) are equal to zero, the coefficient of each order of the Taylor polynomial is equal to zero. We analyze the coefficients of each equation and their relations in order to find smooth functions such that, when they are evaluated on the grid $\{0,\zbar,2\zbar,...,1-\zbar\}$,  they provide an approximation of the fixation probability vector. For each process, BD and DB, the analysis has its peculiarities which will be explained in the following subsections in detail.

\subsection{The continuous approximation candidate of the fixation probability for a DB process}\label{odeDB}

In the DB process, the constant term of the Taylor series of both equations in (\ref{recureq}) vanishes.  
Letting $f_1(x):=q_1(x,0)$ and $f_2(x):=q_2(x,0)$, then the system of coefficients of first-order is equivalent to %the system of ordinary differential equations (ODEs)
\begin{equation}\label{coefone2}
\left\lbrace \begin{array}{l}
{\displaystyle f_2(x)-f_1(x)+f_1'(x)=0}\\\\
{\displaystyle f_1(x)-f_2(x)-f_2'(x)=0}
\end{array}\right. 
\end{equation}

Thus, the solution for this system, considering the initial conditions $f_1(1-\zbar)=1$ and $f_2(0)=0$, is
\begin{equation}\label{solDBfirst}
\left\lbrace \begin{array}{l}
{\displaystyle f_1(x)=\frac{x+1}{2-\zbar}}\\\\
{\displaystyle f_2(x)=\frac{x}{2-\zbar}}
\end{array}\right. 
\end{equation}

Now, using (\ref{solDBfirst}) in the coefficients of second-order and letting $g_1(x):=\frac{\partial q_1}{\partial z}(x,0)$ and $g_2(x):=\frac{\partial q_2}{\partial z}(x,0)$, we obtain
\begin{equation}\label{coeftwo2}
\left\lbrace\begin{array}{l}
{\displaystyle g_2(x)+g_1'(x)-g_1(x)=\frac{1-x \psi_1(x)+\psi_2(x)}{2-\zbar}}\\\\
{\displaystyle g_1(x)-g_2'(x)-g_2(x)=\frac{(x-1)\psi_1(x)+\psi_2(x)-1}{2-\zbar}}
\end{array}\right. 
\end{equation}
The solution for (\ref{coeftwo2}), with initial conditions $g_1(1-\zbar)=0$ and $g_2(0)=0$, is
\begin{small}
\begin{align*}
g_1(x)&:=\frac{(1+x)(1-\zbar)}{(2-\zbar)^2}\int_0^{1-\zbar} 1+\psi_1(k)-(1+2k)\psi_2(k)dk\\\\
&-\frac{x}{(2-\zbar)}\int_{0}^x 1+\psi_1(k)-(1+2k)\psi_2(k)dk-\frac{1+x}{2-\zbar}\int_x^{1-\zbar} 1+(1-2k)\psi_2(k)dk %+\frac{(1+x)}{(2-\zbar)}\int_x^1\left((2k-1)\psi_2(k)-1\right)dk%\\\\
%&\qquad
\end{align*} 
\end{small}and 
\begin{small}
\begin{align*}
g_2(x)&:=\frac{(1-\zbar)x}{(2-\zbar)^2}\int_0^{1-\zbar} 1+\psi_1(k)-(1+2k)\psi_2(k)dk\\\\ 
&- \frac{x}{2-\zbar}\int_x^{1-\zbar}1+(1-2k)\psi_2(k)dk+ \frac{1-x}{2-\zbar}\int_0^x 1+\psi_1(k)-(1+2k)\psi_2(k)dk.
\end{align*}
\end{small}

Let $\bm{\overline{F}}^{DB}$  be the vector such that $\overline{F}^{DB}_i=f_2((i-1)\zbar)+\zbar g_2((i-1)\zbar)$ for $1\leq i \leq N$ and $\overline{F}^{DB}_i=f_1((i-N-1)\zbar)+\zbar g_1((i-N-1)\zbar)$ for $N+1\leq i \leq 2N$. 

We are now ready to state the main theorem of this paper whose proof is given in Appendix~\ref{secA1}.
\begin{theorem}\label{rDB}
Let $\bm{\overline{F}}^{DB}$ be a vector defined as above, then
\[\left\Vert \bm{F}^{DB}-\bm{\overline{F}}^{DB}\right\Vert\leq C\zbar,\] for a constant $C$.
\end{theorem}

\begin{remark}
Note that, for general frequency–dependent fitness, the
discrete fixation vector is unlikely to be available in closed form. On the other hand, its numerical computation requires the solution of
 a system with $2N$ equations, which is impractical for large $N$ even in the case of  constant fitness. Theorem~\ref{rDB} shows that the solution of the continuous ODE system uniformly approximates the entire discrete fixation vector with an $O(1/N)$ error. Consequently, fixation probabilities can be computed from the explicit ODE solution instead of solving the $2N$-equation system. Thus, Theorem~\ref{rDB} provides a rigorously controlled $O(1/N)$-accurate bridge from the discrete model to a solvable continuous ODE formulation in the general frequency-dependent setting.
\end{remark}

\subsection{The continuous approximation candidate of the fixation probability for a BD process}\label{odeBD}
In the BD process, the constant term of the Taylor series for (\ref{recureq}) are
\begin{equation}
\left\lbrace \begin{array}{l}
{\displaystyle -q_1(x,0)+q_2(x,0)=0}\\\\
{\displaystyle q_1(x,0)-q_2(x,0)=0}
\end{array}\right. 
\end{equation}
Therefore, $q_1(x,0)=q_2(x,0)$. Let $f(x):=q_1(x,0)$. Replacing $q_1(x,0)$ and $q_2(x,0)$ by $f(x)$ in coefficients of first-order, we obtain
\begin{equation}
\left\lbrace \begin{array}{l}
{\displaystyle -\frac{\partial q_1}{\partial z}(x,0)+\frac{\partial q_2}{\partial z}(x,0)=0}\\\\
{\displaystyle \frac{\partial q_1}{\partial z}(x,0)-\frac{\partial q_2}{\partial z}(x,0)=0}
\end{array}\right. 
\end{equation}

Thus, $\frac{\partial q_1}{\partial z}(x,0)=\frac{\partial q_2}{\partial z}(x,0)=:g(x)$,  and so the coefficients of second-order are given by
\begin{equation}\label{BDsec}
\left\lbrace \begin{array}{l}
{\displaystyle \frac{1}{2}\left(2f'(x)-\frac{\partial^2 q_1}{\partial z^2}(x,0)+\frac{\partial^2 q_2}{\partial z^2}(x,0)\right)=0}\\\\
{\displaystyle \frac{1}{2}\left(-2f'(x)+\frac{\partial^2 q_1}{\partial z^2}(x,0)-\frac{\partial^2 q_2}{\partial z^2}(x,0)\right)=0}
\end{array}\right. 
\end{equation}

Recall that $h(x):=\frac{\partial^2 q_1}{\partial z^2}(x,0)$. From any of two equations in (\ref{BDsec}) we obtain that $\frac{\partial^2 q_2}{\partial z^2}(x,0)=h(x)-2f'(x)$. So,  the coefficients of third-order can be written as
\begin{small}
\begin{equation}\label{BDthird}
\left\lbrace \begin{array}{l}
{\displaystyle (1+\psi_1(x)-x\psi_2(x))f'(x)+g'(x)+\frac{1}{2}f''(x)-\frac{1}{6}\frac{\partial^3 q_1}{\partial z^3}(x,0)+\frac{1}{6}\frac{\partial^3 q_2}{\partial z^3}(x,0)=0}\\\\
{\displaystyle (\psi_1(x)\!-1+(x-1)\psi_2(x))f'(x)-g'(x)+\frac{1}{2}f''(x)+\frac{1}{6}\frac{\partial^3 q_1}{\partial z^3}(x,0)\!-\frac{1}{6}\frac{\partial^3 q_2}{\partial z^3}(x,0)=0}
\end{array}\right. 
\end{equation}
\end{small}

Summing the equations in (\ref{BDthird}), gives rise to the ODE
\begin{equation}\label{odeBD1}
(2\psi_1(x)-\psi_2(x))f'(x)+f''(x)=0,
\end{equation}
whose solution, with boundary conditions $f(0)=q_2(0,0)=0$ and ~$f(1-\zbar)=q_1(1-\zbar,0)=1$, is
\begin{equation}\label{sol1BD}
f(x)=\frac{\int_0^x e^{-\int_0^s (2\psi_1(r)-\psi_2(r))dr}ds}{\int_0^{1-\zbar} e^{-\int_0^s (2\psi_1(r)-\psi_2(r))dr}ds}.
\end{equation}

Also, letting  $k(x)=\frac{\partial^3 q_1}{\partial z^3}(x,0)$, from (\ref{BDthird}) and (\ref{odeBD1}) we have $\frac{\partial^3 q_2}{\partial z^3}(x,0)= k(x)+6f'(x)((x-1/2)\psi_2(x)-1)-6g'(x)$. We continue the analysis one step more for coefficients of fourth-order and we obtain the ODE

\begin{equation}\label{odeBD2}
\begin{array}{ll}
{\displaystyle (2\psi_1(x)-\psi_2(x))g'(x)+g''(x)}&{\displaystyle =\frac{1}{2} (1-2 x) \psi_2(x) f''(x)+}\\\\
&{\displaystyle f'(x) \left(\psi_1'(x)+\psi_1(x) (-2 x \psi_2(x)+\psi_2(x)-2)\right.}\\\\
&{\displaystyle \left. -\psi_1(x)^2-x \psi_2'(x)+x \psi_2(x)^2+\psi_2(x)\right)}.
\end{array}
\end{equation}
where the r.h.s.  is equal to
\begin{multline*}
\bar{g}(x)=\frac{e^{-\int_{0}^x 2\psi_1(r)+\psi_2(r)dr}}{\int_{0}^{1-\bar{z}}e^{-\int_{0}^s 2\psi_1(r)+\psi_2(r)dr}ds}\times\\
\left(\psi_2(x)+\frac{1}{2}\psi_2^2(x)-2\psi_1(x)-\psi_1^2(x)+\psi_1'(x)-x\psi_2'(x)\right).
\end{multline*}
So,  the solution for (\ref{odeBD2}) with initial condition $g(1-\zbar)=g(0)=0$ is
 \begin{equation*}
g(x)=\int_0^x \left( e^{-\int_0^{s}(2\psi_1(r)-\psi_2(r))dr}\right) \left(C+\int_0^{s}\bar{g}(k)e^{\int_0^{k_2}(2\psi_1(r)-\psi_2(r))dr} dk\right)ds,
\end{equation*}
 where
 
\[
C=\frac{-\int_0^{1-\zbar}\int_0^{s}\bar{g}(k)e^{\int_0^{k}(2\psi_1(r)-\psi_2(r))dr} dk ds}{\int_0^{1-\zbar} e^{-\int_{s}^1(2\psi_1(r)-\psi_2(r))dr}ds}.
\]

In the next sections, we consider the approximate fixation probability vector $\bm{\overline{F}}^{BD}$ such that $\overline{F}^{BD}_i=f((i-1)\zbar)+\zbar g((i-1)\zbar)$ for $1\leq i \leq N$ and $\overline{F}^{BD}_i=f((i-N-1)\zbar)+\zbar g((i-N-1)\zbar)$ for $N+1\leq i \leq 2N$.

\subsubsection{Comparing with the complete graph case}\label{sec:equivalence}
In the BD process,  \cite{chalubsouza16} showed that an approximation of the fixation probability for a large well-mixed population (structured as a complete graph) is equal to
\begin{equation}\label{BDcomplete}
\frac{\int_0^x e^{-\int_0^y (\psi_1(r)-\psi_2(r))dr}dy}{\int_0^1 e^{-\int_0^y (\psi_1(r)-\psi_2(r))dr}dy},
\end{equation}

\noindent and this is a solution of the ODE
\begin{equation}\label{odeC}
(\psi_1(x)-\psi_2(x))f'(x)+f''(x)=0.
\end{equation}

Note that our approximation of the fixation probability for the star graph is quite similar to (\ref{BDcomplete}); the difference is that we have a constant $2$ multiplying the function $\psi_1$. Let $w=1-x$ and $\overline{f}(w):=1-f(1-w)$ be the approximation of the fixation probability for type $A$. Then from (\ref{odeC}), in the complete graph we have
\[
(\psi_2(1-w)-\psi_1(1-w))\overline{f}'(w)+\overline{f}''(w)=0.
\]
Therefore, in the case of a complete graph,  approximating the fixation probability for type $A$ is equivalent to considering the birth fitness function of $B$ as the death fitness function of $A$, and vice versa. However, this equivalence does not hold in the star graph due to the presence of a constant $2$ multiplying $\psi_1$. In fact, in the star graph, we not only need to switch the $\psi_1(x)$ and $\psi_2(x)$ functions but also adjust them accordingly. More precisely, the function $\bar\psi_1(x)$ associated with the birth fitness of type $A$ should be considered as twice of $\psi_2(x)$ function associated with the death fitness of $B$. Similarly, the  $\bar\psi_2(x)$ function related to the death fitness of type $A$ should be half of  $\psi_1(x)$ function related to the  birth fitness of $B$. Figure~\ref{StarEquivalence} provides an example illustrating this equivalence.

\begin{figure}[h!]
\begin{center}
\includegraphics[scale=0.5]{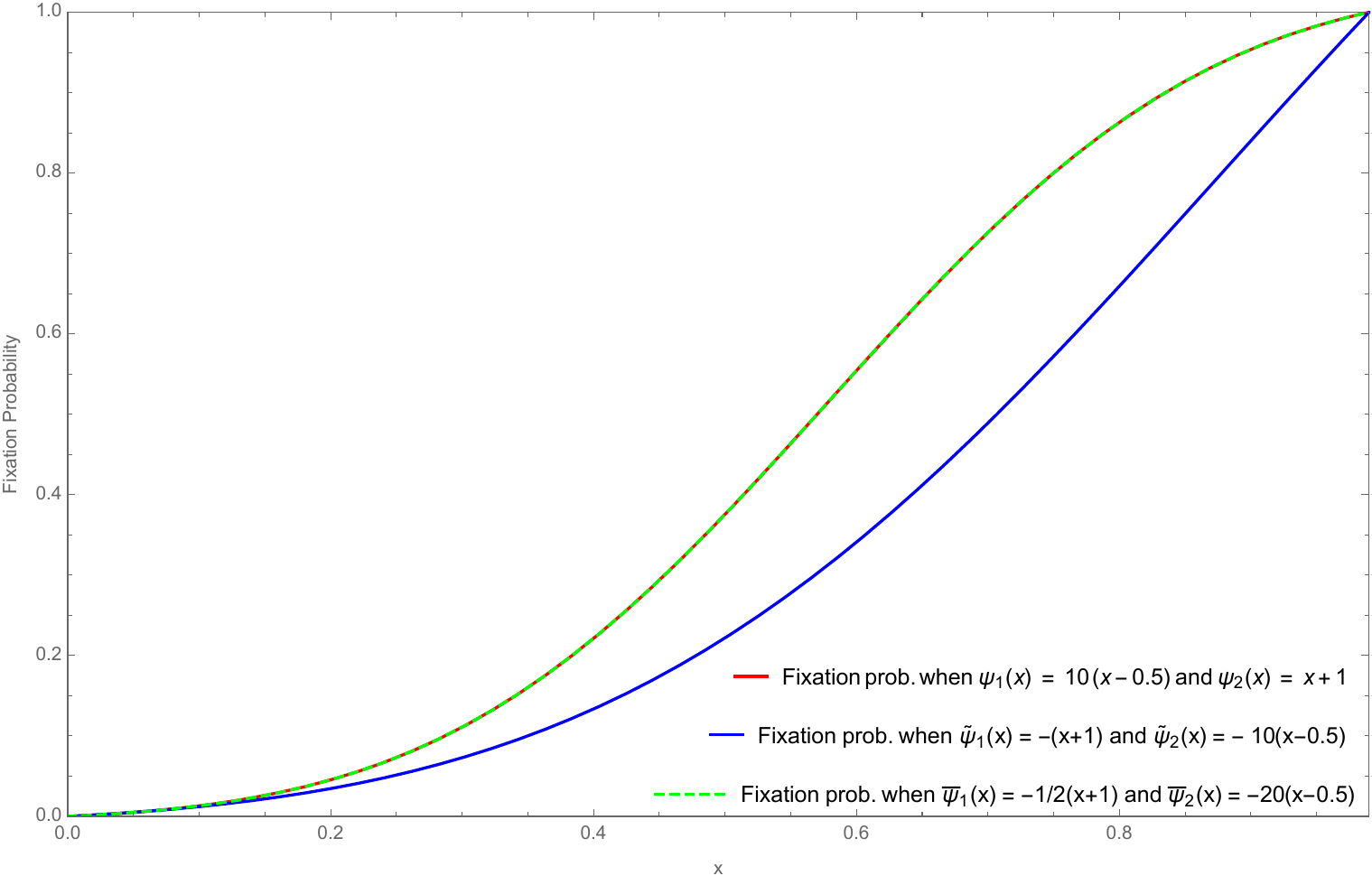}
\end{center}
\vspace*{-5mm}
\caption{ The approximate fixation probability $f(x)$, where $x$ represents the proportion of individuals of type $B$ living at the leaves, is plotted for three distinct cases in the BD process for the population size $N=100$. In the first case (shown in red), the birth fitness function of $B$ is $\varphi_1(x) = 1 + \frac{\psi_1(x)}{N}$, where $\psi_1(x) = 10(x-0.5)$, and the death fitness function of $B$ is $\varphi_2(x) = 1 + \frac{\psi_2(x)}{N}$, where $\psi_2(x) = x+1$. In the second case (shown in blue), the birth fitness function of $B$ is $\tilde\varphi_1(x) = 1 -\frac{\psi_2(x)}{N}$, and the death fitness function of $B$ is $\tilde\varphi_2(x) = 1 - \frac{\psi_1(x)}{N}$. Finally, in the third case (shown in green), which is equivalent to the first case, the birth fitness function of type $B$ is $\bar\varphi_1(x) = 1 -\frac{1}{2} \frac{\psi_2(x)}{N}$, and the death fitness function of type $B$ is $\bar\varphi_2(x) = 1 -2 \frac{\psi_1(x)}{N}$.\label{StarEquivalence}}
\end{figure}

\subsection{Alternative fitnesses parameterizations}

Note that in the DB process, if instead of $\varphi_1^A= 1$ and $\varphi_1^B= 1+\psi_1/N$, we consider $\varphi_1^A=1+\psi_{1}^A/N$ and $\varphi_1^B= 1+\psi_1^B/N$, the system of ODEs (\ref{coefone2}) does not change and (\ref{coeftwo2}) reduces to
\begin{equation}\label{coeftwodif}
\left\lbrace\begin{array}{l}
{\displaystyle g_2(x)+g_1'(x)-g_1(x)=\frac{-x (\psi_1^B(x)-\psi_1^A(x))+\psi_2^{B}(x)-\psi_2^A(x)+1}{2-\zbar}},\\\\
{\displaystyle g_1(x)-g_2'(x)-g_2(x)=\frac{(x-1)(\psi_1^B(x)- \psi_1^A(x))+\psi_2^B(x)-\psi_2^{A}-1}{2-\zbar}}.
\end{array}\right. 
\end{equation}
Therefore,  (\ref{coeftwo2}) is equivalent to (\ref{coeftwodif}) for $\psi_1=(\psi_1^B-\psi_1^A)$ and $\psi_2=(\psi_2^B-\psi_2^A)$. Now, suppose that the fitness of the individual occupying the center is different from the fitness of those occupying the leaves. More explicitly, let $\varphi_1^A(x)=1+\psi_{1}^A(x)/N$ (respectively, $\varphi_1^B(x)=1+\psi_{1}^B(x)/N$) be the birth fitness function for an individual of type $A$ (resp., $B$) when it occupies one of the leaves, and let $\varphi_1^{CA}(x)=1+\psi_{1}^{CA}(x)/N$ (resp., $\varphi_1^{CB}(x)=1+\psi_{1}^{CB}(x)/N$) when it occupies the center, in a population with $xN$ individuals of type $B$.  Similar notation can be easily defined for the death fitness functions. As a result (\ref{coeftwo2}) reduces to
\begin{equation}\label{coeftwodif2}
\left\lbrace\begin{array}{l}
{\displaystyle g_2(x)+g_1'(x)-g_1(x)=\frac{-x (\psi_1^B(x)-\psi_1^A(x))+\psi_2^{CB}(x)-\psi_2^A(x)+1}{2-\zbar}}\\\\
{\displaystyle g_1(x)-g_2'(x)-g_2(x)=\frac{(x-1)(\psi_1^B(x)- \psi_1^A(x))+\psi_2^B(x)-\psi_2^{CA}-1}{2-\zbar}.}
\end{array}\right. 
\end{equation}

Thus, (\ref{coeftwodif}) is a particular case of (\ref{coeftwodif2}), when $\psi_2^{CA}=\psi_2^{A}(x)$ and $\psi_2^{CB}=\psi_2^{B}(x)$. 

This is similar for the BD process. If instead of $\varphi_1^A= 1$ and $\varphi_1^B= 1+\psi_1/N$ we consider $\varphi_1^A=1+\psi_{1}^A/N$ and $\varphi_1^B= 1+\psi_1^B/N$,  then (\ref{odeBD1}) is equivalent to
\begin{equation}\label{odeBD3}
[2(\psi_1^B(x)-\psi_1^A(x))-(\psi_2^B(x)-\psi_2^A(x))]f'(x)+f''(x)=0.
\end{equation}

\noindent Also, in the case that the individuals at the center and the leaves have different fitness functions, (\ref{odeBD1}) is equivalent to
\begin{equation}\label{odeBD4}
[(\psi_1^B(x)-\psi_1^A(x))+ (\psi_1^{CB}(x)-\psi_1^{CA}(x))-(\psi_2^B(x)-\psi_2^A(x))]f'(x)+f''(x)=0.
\end{equation}

So, (\ref{odeBD3}) is a particular case of (\ref{odeBD4}), when $\psi_1^{CA}=\psi_1^{A}(x)$ and $\psi_1^{CB}=\psi_1^{B}(x)$.

\subsection{Comparing the BD and DB processes on Star graph}\label{BDvsDB}

While the final forms of the approximations for the BD and DB processes may look quite different, they are in fact derived using the same underlying methodology. The divergence in the final formulas arises from the opposite ways in which the two update rules interact with the geometry of the star graph.

The star can be viewed as a structured population in which the $N-1$ leaves only interact indirectly through the central vertex. In the BD process, a mutant leaf can increase the leaf mutant count only after two successive reproduction events: first, the mutant leaf must be selected to reproduce, and in the following step, the center must in turn be chosen to reproduce. These consecutive, coordinated “leaf $\rightarrow$ center $\rightarrow$ leaf” events occur with probability one but require a long waiting time. Since the center has the same individual reproduction probability (depending on fitness) as any single leaf, it is much less likely to be chosen for reproduction than the leaves collectively. As a result, it often changes type many times before the number of mutants among the leaves increases or decreases. Fixation is therefore governed by these long-waiting coordinated events.

In contrast, the DB process reverses the order of events: death occurs first, and the replacement individual is chosen proportionally to fitness. When the center carries a mutant, it can repeatedly replace neighboring leaves before it is itself removed, exerting a strong directional influence on the periphery. Mutant expansion is thus dominated by frequent “center $\rightarrow$  leaf” replacements, rather than long waiting periods for two-step sequences expected in the BD process. It often takes a long time for the center itself to change type, and between such switches, it drives the replacement of surrounding leaves one after another. The resulting continuous approximation reflects this asymmetry and thus differs significantly from its BD counterpart.

\section{Fitness functions given by $2$-player games}\label{sec:2player}
In this section,  we focus on the BD process on a star graph in the particular case where the fitnesses are linear functions of the frequencies, determined by $2$-player games with weak-selection.  We follow~\cite{chalubsouza16} to show that the asymptotic qualitative behavior of a population structured as a star graph is the same as that structured as a complete graph.  In fact, for large $N$, if we only consider the leaves, we expect the behavior of the population in the star graph be quite similar to that in the complete graph, as the center of the star has the role of connecting leaves, i.e. leaves interact with each other through the center.

Consider the case that $\psi_2(x)=0$ and $\psi_1=(\psi_1^B-\psi_1^A)$, where $\psi_1^B$ and $\psi_1^A$ are given by the $2\times 2$ positive pay-off matrix

\begin{center}
\begin{tabular}{c|cc}
& $A$ &$B$\\\hline
$A$& $a$ &$b$\\
$B$ &$c$& $d$
\end{tabular}
\\[6pt]
\end{center}
that is,  $\psi_1^A(x)= a x+b(1-x)$ and $\psi_1^B(x)= c x+d(1-x)$. This is equivalent to write $\psi_1$ as $\psi_1(x)= \gamma (x-x^*)$, where $\gamma=(-a+b+c-d)$ and $x^*=(b-d)/\gamma$.

If $\gamma<0$ and $0<x^*<1$, we have the coexistence case. If $\gamma>0$ and $0<x^*<1$, the game is called coordination (or continuation game), when the two types have the same or corresponding fitnesses. If $\psi_1>0$, it is said that type $B$ dominates type $A$, and if $\psi_1<0$, type $A$ dominates type $B$. 

For $\kappa=\kappa_N>0$,  indicating the selection intensity, let us now modify our birth fitness function by $\varphi_1=1+\kappa^{-1}\psi_1/N$.  In the case that $\kappa^{-1}\gg 1$,  let $\theta_s=s\psi_1$ for $s=1,2$, and let $\phi_{1}^\kappa$ be the approximate fixation probability for the complete graph given in (\ref{BDcomplete}) (\cite{chalubsouza16}) and $\phi_{2}^\kappa$ be our approximation for the star graph in the BD case with the above fitness functions.  Following the same lines of argument in~\cite{chalubsouza16}, for the dominance case, if $\theta_s>0$, $B$ is dominant and
\begin{equation}\label{domB}
\phi_{s}^\kappa (x)=1-\exp(-\theta_s(0)x/\kappa)+O(\kappa).
\end{equation}
In fact, type $B$ dominates in the star graph faster than in the complete graph.  Similarly, if $\theta_s<0$, $A$ is dominant and
\begin{equation}\label{domA}
\phi_{s}^\kappa (x)=\exp(\theta_s(1)(1-x)/\kappa)+O(\kappa).
\end{equation}
So, in the star graph, type $A$ dominates slower than in the complete graph. An example is given in Figure~\ref{figdom}. 

\vspace*{2mm}
\begin{figure}[htbp]
\begin{center}
\begin{tabular}{cc}
\includegraphics[scale=0.72]{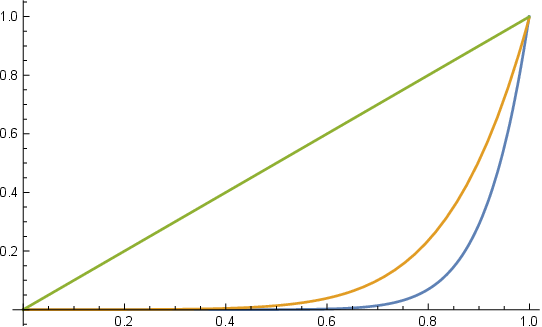}
&\includegraphics[scale=0.72]{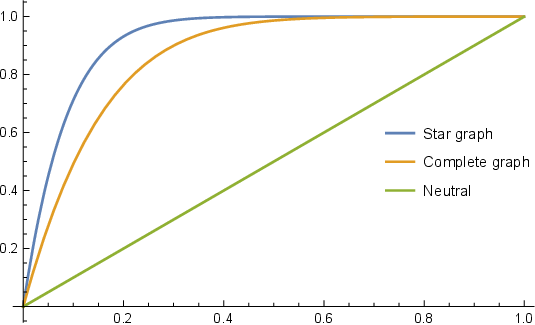}
\end{tabular}
\end{center}
\vspace*{-5mm}
\caption[Examples of dominance]{Approximations for the fixation probability in the star graph, in blue, and in the complete graph in orange.  The fixation probability for the neutral case is given in green.  On the left: $N=1000$, $\kappa^{-1}=10$, $\psi_1(x)=(x-1.5)$ and $\psi_2=0$; $A$ dominates.  On the right: $N=1000$, $\kappa^{-1}=10$, $\psi_1(x)=(x+0.5)$ and $\psi_2=0$; $B$ dominates.\label{figdom}}
\end{figure}

In the coexistence case, if $\int_0^1 \theta_s(r)dr\ll -\kappa$, the asymptotic approximation is given by (\ref{domB}). If $\int_0^1 \theta_s(r)dr\gg \kappa$, the asymptotic approximation is given by (\ref{domA}). Finally,  if $\int_0^1 \theta_s(r)dr \sim \kappa$,  we have
\begin{equation}\label{coe}
\phi_{s}^\kappa (x)=\frac{C}{C+\lambda} \exp(\theta_s(1)(1-x)/\kappa)+\frac{\lambda}{C+\lambda}(1-\exp(-\theta_s(0)x/\kappa))+O(\kappa),
\end{equation} with $\theta_s(0)> 0> \theta_s(1)$,  where $C=\exp(\kappa^{-1}\int_0^1 \theta_s(r)dr)$, and $\lambda=\vert\theta_s(1)\vert/\theta_s(0)$. 

In the coordination case, $\theta$ also has a unique root $x^*$, with $\theta'(x^*)>0$, and we have
\begin{equation}\label{coo}
\phi_{s}^\kappa (x)=\frac{\NN\left(\sqrt{\frac{\theta_s'(x^*)}{\kappa}}(x-x^*)\right)-\NN\left(-\sqrt{\frac{\theta_s'(x^*)}{\kappa}}x^*\right)}{\NN\left(\sqrt{\frac{\theta_s'(x^*)}{\kappa}}(1-x^*)\right)-\NN\left(-\sqrt{\frac{\theta_s'(x^*)}{\kappa}}x^*\right)}+O(\sqrt{\kappa})
\end{equation}
where $\NN(x)=\frac{1}{\sqrt{2\pi}}\int_{- \infty}^x e^{-y^2/2}dy$ is the normal cumulative distribution function. An example is given in Figure~\ref{figcoo}.

When $\kappa=1$, 
\[ 
\phi_{s}''(x)=\frac{-s\psi_1(x)e^{-s\int_0^r \psi(r)dr}}{\int_0^1 e^{-s\int_0^r \psi (r)dr}}
\]
implies that $A$ is dominant for a convex function $\phi_{s}$, and $B$ is dominant for a concave function $\phi_{s}$.  If $\phi_{s}$ has an inflection point 
with the concave part coming first (on the left of the inflation point) and the convex part coming next (on the right of the inflation point), the game is a coexistence game. Finally, if $\phi_{s}$ has an inflection point with the convex part coming first (on the left) and the concave part coming next (on the right), the game is a coordination game. 

In the DB process, if $\kappa$ is of order less than $N^{-1}$, then the behavior of the population does not essentially depend on the fitness functions.

\begin{figure}[htbp]
\begin{center}
\begin{tabular}{cc}
\includegraphics[scale=0.72]{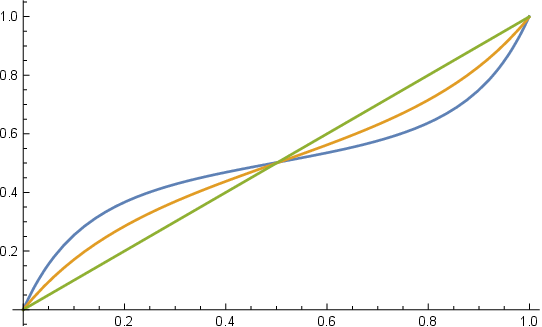}
&\includegraphics[scale=0.72]{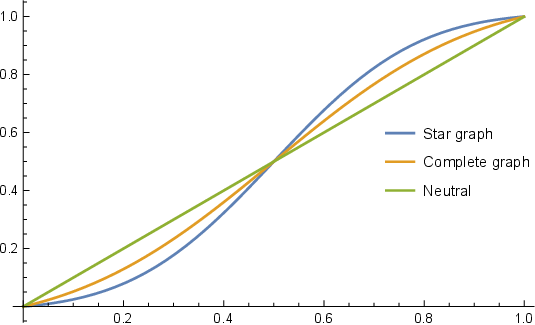}
\end{tabular}
\end{center}
\vspace*{-5mm}
\caption[Examples of coordination and coexistence games]{Approximations for the fixation probability in the star graph, in blue, and in the complete graph in orange, for the BD process. The fixation probability for the neutral case is given in green.  On the left: $N=1000$, $\kappa^{-1}=10$, $\psi_1(x)=(0.5-x)$ and $\psi_2=0$; we have the coexistence game.  On the right: $N=1000$, $\kappa^{-1}=10$, $\psi_1(x)=(x-0.5)$ and $\psi_2=0$; we have the coordination game. \label{figcoo}}
\end{figure}

\section{Invasion probability}\label{sec:IP}

Consider the approximate fixation probability vectors for the star graph introduced in Section~\ref{sec:odecandidate}. Recall that $\bar z=1/N$, and let $\phi_{1}(\bar{z})$ and $\phi_{2,\rho}(\bar{z})$ denote the approximate invasion probabilities of a single mutant on a complete graph (given in (\ref{BDcomplete})) and a star graph of the same size, respectively, where the sub-index $\rho=\rho_N$ in the latter denotes the probability that a single mutant in the population occupies the center of the star. Here, by a complete graph, we mean a graph in which every pair of vertices is connected by an edge, and, additionally, there is a loop at each vertex.

In this section, we analyze the limiting behavior of the ratio of invasion probabilities in the star and complete graphs under weak selection, for both DB and BD processes. Specifically, we study the limit
\[
\lim_{\bar{z} \to 0} \frac{\phi_{2,\rho}(\bar{z})}{\phi_{1}(\bar{z})}
\quad (\text{equivalently, as} \ N \to \infty),
\]
for both the BD and DB processes, and under different regimes for the probability $\rho$ that the initial mutant appears at the center of the star. This limit characterizes how the population structure influences the success of rare mutants, depending not only on the update rule (BD or DB) and the spatial location of the initial mutation, but also on how the probability $\rho$ is chosen. Based on the value of this limit, the star graph may act as an amplifier or suppressor of selection, or have the same asymptotic invasion probability as the complete graph.

Formally, for a complete graph with \(N\) vertices, let \(\tilde h_N(i)\) denote the fixation probability of the mutant when the initial configuration includes exactly \(i\) mutants and \(N-i\) residents. We say that the mutant type $B$ is asymptotically stronger (weaker, respectively) than the resident type $A$, denoting by $B \succ A$ ($A\succ B$, respectively, if $\liminf_N \tilde h_{2N}(N)>1/2$ ($\limsup_N \tilde h_{2N}(N)<1/2$, respectively). We write $A\sim B$ if $\lim_{N\rightarrow \infty} \tilde h_{2N}(N)=1/2$. 

To see a particular case of this, recall that under the reduced Markov chain on the complete graph, the probability of increasing the number of mutants from \( k \) to \( k+1 \) is  
\[
\frac{\varphi_1(k/N)}{\varphi_1(k/N) + \varphi_2(k/N)},
\]  
and the probability of decreasing from \( k \) to \( k-1 \) is  
\[
\frac{\varphi_2(k/N)}{\varphi_1(k/N) + \varphi_2(k/N)}.
\]  

This means that the probability of increasing mutants exceeds that of decreasing them for all states if and only if \( \psi_1 - \psi_2 > 0 \). In this case we say $B$ is stronger than $A$ for all configurations. This is stronger than the ``stronger'' concept we defined above, and clearly implies $B\succ A$. Similarly if \( \psi_1 - \psi_2 < 0 \) we say $A$ is stronger than $B$ for all configurations which implies $A\succ B$. Finally, we say the process is neutral if \( \psi_1 - \psi_2 = 0 \). Under neutrality, we have $A\sim B$.

Suppose that the limit  
$
\beta=\beta(\psi_1,\psi_2) := \lim_{\bar{z} \to 0} \phi_{2,\rho}(\bar{z})/\phi_1(\bar{z})
$
exists. We say that the star graph is an \emph{amplifier} of selection under the BD (or DB) process if $B\succ A$ and \( \beta > 1 \), or if $A\succ B$ and \( \beta < 1 \). That is, the star structure strengthens the impact of selection, making the fitter type even more likely to fix than under the complete graph.

Conversely, we say it is a \emph{suppressor} of selection if $B\succ A$ and \( \beta < 1 \), or if $A\succ B$ and \( \beta > 1 \).
We say the star graph is 
\emph{asymptotically invasion probability preserving} (or a preserver for short), if \( \beta = 1 \).
This usage is consistent with the fact that, for constant fitness, the class of processes that are invasion probability preserving  can be larger than the isothermal graphs as observed in \cite{pattni2015evolutionary}.

This can be summarized as follows. Assume that $h_{2N}(N)$ converges as $N\to\infty$ and let \[\tau=\tau(\psi_1,\psi_2)=\sign\left(\lim_{N\rightarrow \infty} \tilde h_{2N}(N)- \frac{1}{2}\right), \]
where $\sign(x)=1$ if $x>0$, $\sign(x)=-1$ if $x<0$, and $\sign(x)=0$ if $x=0$. We call $\beta^\tau$ the magnitude of amplification-suppression.
\begin{definition}
Suppose that the limit $\beta$ exists. We say that the star graph is:
\begin{enumerate}
    \item[i)] an \emph{amplifier} of selection if $\beta^\tau>1$,
    \item[ii)] a \emph{suppressor} of selection if $\beta^\tau<1$,
    \item[iii)] and \emph{asymptotically invasion probability preserving} if \( \beta = 1 \).
\end{enumerate}
\end{definition}
 This notion naturally extends beyond the star graph: the same classification can be applied to any structured population, as long as a suitable limit of the ratio of fixation (or invasion) probabilities is well-defined.

Note that $(iii)$ in the above definition is stronger than $\beta^\tau=1$, as the latter is equivalent to either $\beta=1$ (being \aipp) or $\tau=0$ (being neutral). Also note that, for both BD and DB processes, there are cases in which the process with fitness functions $\psi_1$ and $\psi_2$ is neutral on the complete graph (well-mixed population), while it is not neutral on the star graph. In this case, we may have $\beta>1$ or $\beta<1$.

We begin with a brief explanation of the invasion probability in the complete graph in Subsection~\ref{invcomplete}, which serves as the baseline for comparison. We then analyze the ratio of invasion probabilities between the star and the complete graph, starting with the DB process in Subsection~\ref{ratiodb}, followed by the BD process in Subsection~\ref{ratiobd}.

\subsection{Invasion probability in the complete graph}\label{invcomplete} 
We assume that the unique mutant resides at the center of the star with probability \(\rho = \rho_N \in [0,1]\). Since the center is a distinguished vertex in the star, it is natural to compare the invasion probability in the star with that in the complete graph under the assumption that a distinguished vertex is chosen for the mutant to reside with probability \(\rho\). More precisely, consider a bijective map $\zeta$ from the star graph with $n-1$ leaves to the complete graph with $n$ vertices. The image of the center of the star under $\zeta$ becomes a distinguished vertex in the complete graph, which we refer to as the \emph{root} of the complete graph under $\zeta$.
The distribution 
\[
(\rho, (1-\rho)/(N-1), \ldots, (1-\rho)/(N-1))
\]
for selecting the site of a single initial mutant in the star graph can be pushed forward through $\zeta$ to define a corresponding distribution on the complete graph: the root is selected with probability $\rho$, and each of the other vertices with probability $(1-\rho)/(N-1)$.
However, due to the symmetry of the complete graph, the specific choice of $\zeta$ and $\rho$ does not affect the invasion probability of a single mutant. In other words, the invasion probability on the complete graph is independent of the particular choices of $\zeta$ and $\rho$.
Furthermore, the invasion probability of a single mutant in a complete graph of size $n$ is the same under both the BD and DB processes.

For a complete graph with \(N\) vertices, recall that \(\tilde h(i)=\tilde h_N(i)\) denotes the fixation probability of the mutant when the initial configuration includes exactly \(i\) mutants and \(N-i\) residents. For \(i = 1, \dots, N-1\), we have
\[
\tilde h(i) = p_i \cdot \tilde h(i+1) + q_i \cdot \tilde h(i-1),
\]
with boundary conditions \(\tilde h(N) = 1\), \(\tilde h(0) = 0\), where
\begin{equation}\label{jump-prob-complete}
    p_i^{N} = p_i :=\frac{\varphi_1(i/N)}{\varphi_1(1/N)+\varphi_2(i/N)}=\frac{1 + \psi_1\left(\frac{i}{N}\right)/N}{2 + \left(\psi_1\left(\frac{i}{N}\right) + \psi_2\left(\frac{i}{N}\right)\right)/N}
\end{equation}
is the probability that, conditional on having \(i\) mutants in the current state, a resident dies and is replaced by an offspring of a mutant, so that the number of mutants increases by 1. Similarly, for \(i = 1, \dots, N-1\), \(q_i^{N} = q_i := 1 - p_i\) is the probability that the Markov chain jumps from state \(i\) mutants to \(i+1\) mutants at the first jump.

From (1), letting \(\Delta_j = \tilde h(j) - \tilde h(j-1)\), we get \(\Delta_{j+1} = (q_j / p_j) \Delta_j\). Hence,
\[
\Delta_{j+1} = \left( \prod_{r=1}^j \frac{q_r}{p_r} \right) \tilde h(1),
\]
and then
\[
1 = \tilde h(N) = \tilde h(1) + \sum_{j=1}^{N-1} \Delta_{j+1} = \tilde h(1) \left( 1 + \sum_{i=1}^{N-1} \prod_{r=1}^i \frac{q_r}{p_r} \right).
\]

Therefore, for \(i = 1, \dots, N\),
\begin{equation}
\tilde h_N(i) = \frac{1 + N \cdot \bar{J}_{i,N}}{1 + N \cdot \bar{J}_{N,N}}, \tag{2}
\end{equation}
where
\[
\bar{J}_{i,N} = \frac{1}{N} \sum_{j=1}^{i-1} \prod_{r = 1}^{j} \frac{q_r}{p_r}, \quad i = 1, \dots, N.
\]

Note that, by convention, we assume \( \sum_{j=1}^0 x_j= 0\) for any sequence $(x_j)$.
The next step is to find the asymptotics for \(\tilde h(i)\). 
For \(u, s \in [0,1]\), let
\begin{align*}
\chi(u) &:= \psi_2(u) - \psi_1(u), \\
\xi(y) &:= \exp\left( \int_0^y \chi(u)\, du \right), \\
\mathcal{J}(s) &:= \int_0^s \xi(y)\, dy.
\end{align*}
\begin{proposition}\label{fixation-Complete}
For any \(i = 1, \dots, N\), we have
\[
\tilde h_N(i) = \frac{\mathcal{J}(i/N)}{\mathcal{J}(1)} + O(N^{-1}).
\]
\end{proposition}
\begin{proof}
    See Appendix~\ref{AppendixB}.
\end{proof}
\begin{remark}
In particular, \[\phi_{1}(\bar z):=\frac{\mathcal{J}(1/N)}{\mathcal{J}(1)},\] and the actual invasion probability of a mutant, in the complete graph of size $N$, is $\phi_1(\bar z)+O(N^{-1})$.
\end{remark}
\begin{remark}
If \(i = i(N)\) is such that \(i/N \to s \in [0,1]\), then
\[
\tilde h_N(i) = \frac{\mathcal{J}(s)}{\mathcal{J}(1)} + O(N^{-1}).
\]
\end{remark}

In addition, it is clear that
\[
N \tilde h_N\left( 1 \right) \to \left( \mathcal{J}(1) \right)^{-1}
\quad \text{as } N \to \infty.
\]

\subsection{Invasion in the star graph: DB process}\label{ratiodb}
In the DB process, a large population has a much higher chance to resist against the invasion of a mutant occupying a leaf than  one occupying the center.  In fact, the fixation probability of a single mutant starting at a leaf is approximately $1/(2N)$. On the other hand, if a mutant invades the center, it has approximately $1/2$ chance to fix. Letting $\rho$ be the probability that a single mutant in the population occupies the center, our approximate invasion probability in the DB process is
\[
\phi_{2,\rho}(\zbar)=\rho\left(\frac{N}{2N-1}+\frac{g_1(0)}{N}\right) +(1-\rho)\left(\frac{1}{2N-1}+\frac{g_2(1/N)}{N} \right).
\]
From Theorem~\ref{rDB}, the actual invasion probability of a mutant is given by
\[
\tilde \phi_{2,\rho}(\bar z):=\rho F_{N+1}+(1-\rho)F_2=\rho P_{1,0}+(1-\rho) P_{2,\bar Z}=\phi_{2,\rho}(\bar z)+O(N^{-1}).
\]
We observe that, in the DB process, the limit of the ratio of invasion probabilities between the star and the complete graph diverges as $N \to \infty$, provided that $\rho$ is constant and strictly positive (i.e., independent of $N$). That is,
\[
\lim_{N \to \infty} \frac{\tilde\phi_{2,\rho}(\bar z)}{\tilde h_N(1)}=\lim_{N \to \infty} \frac{\phi_{2,\rho}(\bar z)}{\phi_{1}(\bar z)} = \infty \quad \text{if } \rho > 0 \text{ is constant}.
\]

This shows that if $B\succ A$, the star graph functions as an amplifier of selection under the DB process for any constant $\rho > 0$ and for general fitness functions $\psi_1$ and $\psi_2$. Similarly, it is a suppressor of selection, if $A\succ B$ and $\rho>0$ is constant.

By contrast, in prior studies where the initial mutant is assumed to be placed uniformly at random across all vertices, the probability of occupying the center is $\rho = 1/N$. Under this regime, the star graph becomes a suppressor of selection, if $B\succ A$ and an amplifier of selection if $A\succ B$.

While the limiting uniform placement case has been studied before, general regimes for $\rho = \rho_N$ appear not to have been systematically analyzed. To better understand the transition between these regimes, we assume $\rho_N\to 0$ as $N\to \infty$, and study different scaling regimes of $\rho_N$. First, consider $\rho_N = O(1/N)$, i.e., $N \rho_N \to \alpha \in [0,\infty)$. In this regime, the ratio between the star and complete graph fixation probabilities becomes
\begin{eqnarray}\label{DB-invasion-proportion}
\lim_{N \to \infty} \frac{\tilde\phi_{2,\rho}(\bar z)}{\tilde h_N(1)}=
\frac{\alpha + 1+\tilde C}{2}
\int_0^1 e^{ -\int_0^y (\psi_1(u) - \psi_2(u)) \, du } dy \nonumber \\
\times
\lim_{\zbar \to 0}
\frac{1}{N \int_0^{\zbar} e^{ -\int_0^y (\psi_1(u) - \psi_2(u)) \, du } dy } \nonumber\\
=  \frac{\alpha + 1+\tilde C}{2}
\int_0^1 e^{ -\int_0^y (\psi_1(u) - \psi_2(u)) \, du } dy,
\end{eqnarray}
for some constant $\tilde C=\tilde C(\psi_1,\psi_2)$, where $C'$ comes from the error of the approximate invasion probability when the unique mutant resides in a leaf, that is, the error is of order
\[
\frac{1}{2N-1}+\frac{g_2(1/N)}{N} - P_{2,\bar z}=\frac{\tilde C}{N}+o(N^{-1}).
\]
Based on numerical examples, we conjecture that $\tilde C=0$ and the error of the invasion probability is of order $o(N^{-1})$.

Although the expression on the right of (\ref{DB-invasion-proportion}) is a finite value, it can be greater than, equal to, or smaller than $1$ depending on the values of $\tilde C$ (that we conjecture is $0$) and $\alpha$ and the behavior of the exponential integral. Therefore, in this regime, the star graph may act as an amplifier, a suppressor, or \aipp, depending on the fitness functions $\psi_1$ and $\psi_2$.

Now consider the case $\rho_N \gg 1/N$, i.e., $N \rho_N \to \infty$ (for instance, $\rho_N = 1/\sqrt{N}$). In this regime, as $N \to \infty$, we have

\[
\lim_{N \to \infty} \frac{\tilde\phi_{2,\rho}(\bar z)}{\tilde h_N(1)}=\lim_{N \to \infty}\frac{\rho_N N\int_0^1 e^{ -\int_0^y (\psi_1(u) - \psi_2(u)) \, du } dy}
{2N\int_0^{1/N} e^{ -\int_0^y (\psi_1(u) - \psi_2(u)) \, du } dy}=\infty,
\]\\[2pt]

and hence the star graph acts as an amplifier of selection in this scaling regime if $B\succ A$, and as a suppressor if $A\succ B$.

This result shows that even sublinear probabilities for the initial mutant to occupy the center of the star graph---such as $\rho_N = 1/\sqrt{N}$---are sufficient to generate amplification in large populations if $B\succ A$. The analysis demonstrates a nontrivial asymptotic transition: any scaling of $\rho_N$ that decays more slowly than $1/N$ leads to amplification if $B\succ A$. This includes the case where $\rho_N$ is constant and positive (i.e., $\rho_N = \rho > 0$), as such a choice also satisfies $N \rho_N \to \infty$.

From this analysis, we can deduce the following theorem.

\begin{theorem}\label{thm:2}
Consider the DB process on the star graph with $N$ vertices. Suppose $B\succ A$, and let $\rho_N$ be such that $\alpha = \lim_{N \to \infty} N \rho_N \in [0,\infty]$. Then there exists a finite $C > 0$ such that
\begin{enumerate}
    \item[i)] If $\alpha > C$, the star is an amplifier of selection;
    \item[ii)] If $\alpha = C$, the star is \aipp;
    \item[iii)] If $\alpha < C$, the star is a suppressor of selection.
\end{enumerate}   
\end{theorem}
\begin{remark}
    The case $A\succ B$ can be treated similarly.
\end{remark}
\begin{remark}\label{alpha-switch-value}
    The constant $C$ depends on $\psi_1$ and $\psi_2$, and can be computed if we have more information about the error $\tilde \phi_{2,\rho}(\bar z)- \phi_{2,\rho}(\bar z)$. The numerical values of $C$ are given later in this section. As mentioned in the discussion before the theorem, we conjecture $\tilde C=0$ in (\ref{DB-invasion-proportion}). Assuming that $\tilde C$ vanishes, we can compute the exact value of $C$ in Theorem~\ref{thm:2}. More explicitly, from (\ref{DB-invasion-proportion}) $C$ should satisfy $
    (C+1)\mathcal{J}(1)/2=1,$
    implying $$C=(2-\mathcal{J}(1))/\mathcal{J}(1).$$.
\end{remark}
\begin{remark}
    Under the neutrality assumption in the complete graph, i.e. $\psi_1-\psi_2=0$, we have $C=1$, that is, for $\rho=1/N$ the star graph becomes \aipp.
\end{remark}

To illustrate the implications of Theorem~\ref{thm:2} and to gain intuition about the dependence of invasion success on model parameters, we begin with a simple example where the difference of the fitness functions is constant, that is, $\psi_1(x) -\psi_2(x) = r$, for $x\in [0,1]$, with $r$ constant. This setting represents a linear selection gradient independent of frequency and provides a useful reference for more complex dynamics. As mentioned before, if $r>0$, then $B\succ A$; if $r<0$, then $A\succ B$; and if $r=0$, then $A\sim B$. In other words, we have $\tau=\sign(r)$ . We also consider a scaling of the initial mutant placement probability $\rho_N = \alpha / N$, where $\alpha \geq 0$. This corresponds to a regime where the center receives a small but non-negligible probability of mutation, and allows us to explore how the value of $\alpha$ affects whether the star graph acts as an amplifier or suppressor. In this case, $\beta=(1-e^{-r})(1+\alpha)/(2r)$ (see Figure~\ref{Figure-1-DB}).

\begin{figure}[htbp]
\begin{center}
\includegraphics[scale=0.52]{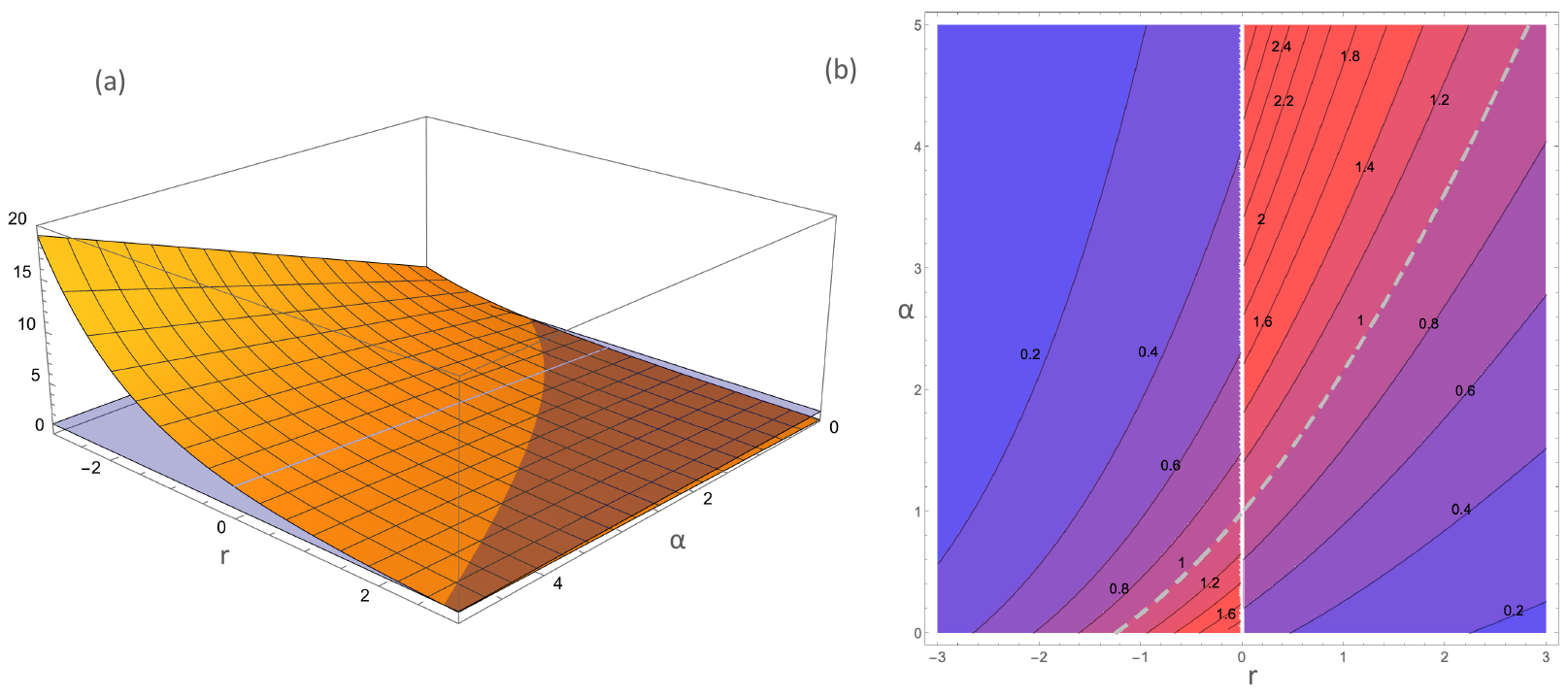}
\end{center}
\vspace*{-5mm}
\caption{
Limit of the ratio of invasion probabilities $\beta=\lim_{\bar{z} \to 0}\phi_{2,\rho} / \phi_1$ in the DB process for $\psi_1(x) - \psi_2(x) = r$, with $\rho_N = \alpha / N$, over the parameter ranges $-3 \leq r \leq 3$ and $0 \leq \alpha \leq 5$. (a) 3D surface plot of the limit ratio $\beta$ in orange. (b) Contour plot of $\beta^\tau$, for $r\neq 0$, where the color scale indicates whether the star graph acts as an amplifier (red) or a suppressor (blue) relative to the complete graph. The dashed line correspond to $\beta=1$.
\label{fig:constantfit3D}\label{Figure-1-DB}}
\end{figure}

Figure~\ref{fig:constantfit3D} (a) shows the three-dimensional surface of the ratio of the limit $\beta=\lim_{N\rightarrow \infty}\phi_{2,\rho} / \phi_1$ as a function of $(r,\alpha)$, and Figure~\ref{fig:constantfit3D} (b) shows the corresponding contour plot for $\beta^\tau$, where as before, $$\tau=\sign\left(\frac{\mathcal{J}(1/2)}{\mathcal{J}(1)}-\frac{1}{2}\right).$$ The results reveal a transition between suppressor and amplifier behavior that depends jointly on both parameters. Specifically, for each fixed value of $r$, there exists a critical value of $\alpha$ at which the star graph switches from acting as a suppressor (or amplifier) to acting as an amplifier (or suppressor, respectively). This transition occurs along the curve
\[
\alpha=\alpha(r)= \frac{2r+e^{-r}-1}{ 1-e^{-r}},
\]
which corresponds to the level set where the invasion probability ratio equals one, i.e. $\beta=1$. This determines the value of $C$ in Theorem~\ref{thm:2} and Remark~\ref{alpha-switch-value}, and corresponds to the regimes under which the DB process on the star graph is \aipp. Furthermore, this analytical threshold is derived from Theorem~\ref{thm:2} and matches the empirical transition curve in the plots. Even in this simple case, the star graph exhibits a rich spectrum of structural effects depending on selection intensity and initial mutation placement.

As a next example, we now consider the class of frequency-dependent fitness functions derived from  $2 \times 2$ evolutionary games under weak selection, as discussed in Section~\ref{sec:2player}. In this setting, the fitness difference $\psi_1$ takes the form $\psi_1(x) = \gamma(x - x^*)$ and $\psi_2=0$, with $\gamma, x^*\in \mathbb R$. 
We also let $\rho_N=1/N$. This parametrization arises from a payoff matrix under weak selection, and it captures all classical game types: dominance, coexistence, and coordination.

\begin{figure}[h!]
\begin{center}
\includegraphics[scale=0.55]{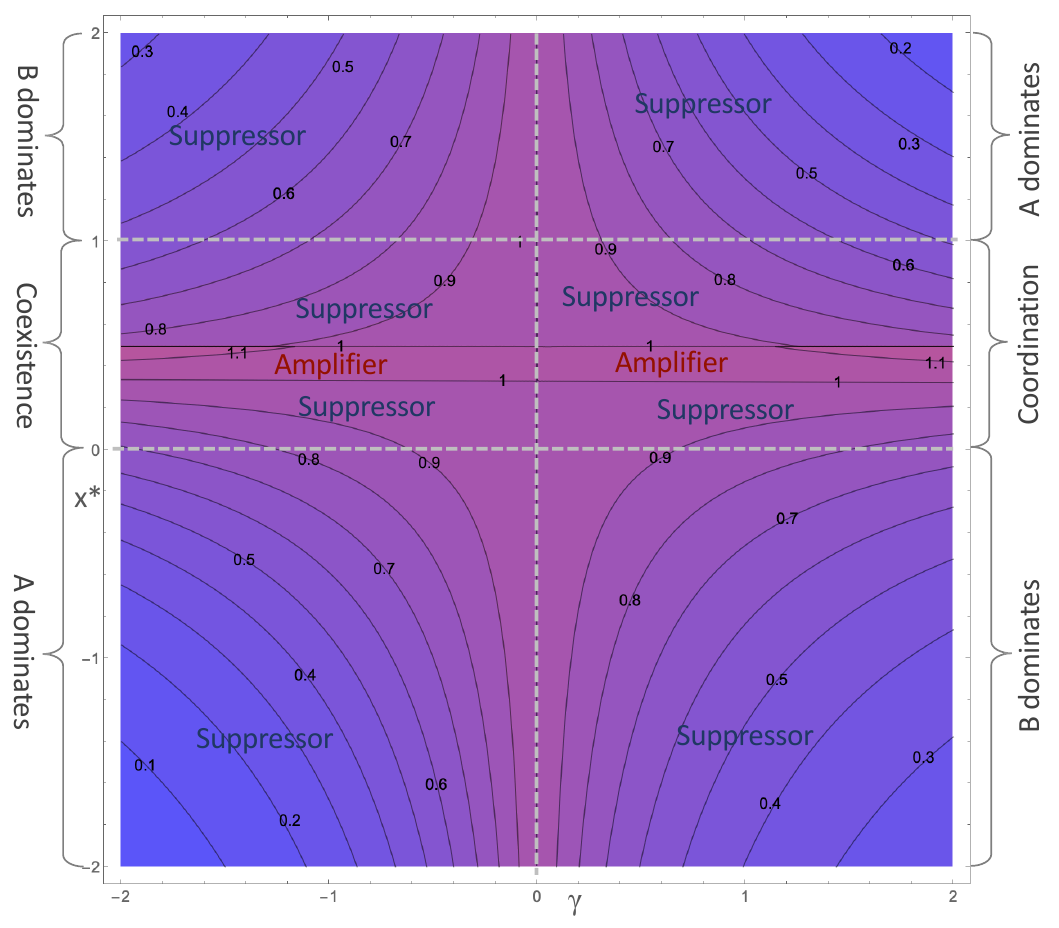}
\end{center}
\vspace*{-5mm}
\caption{
Contour plot of $\beta^\tau$ for the DB process, where the fitness functions are derived from a two-player game with weak selection. We consider $\psi_1(x) = \gamma(x - x^*)$, $\psi_2(x) = 0$ and $\rho_N=1/N$, a parametrization that captures all classical game types (dominance, coexistence, and coordination) arising from $2 \times 2$ payoff matrices. The domain spans $-2 < \gamma < 2$ and $-2 < x^* < 2$, and the color scale indicates whether the star graph acts as an amplifier (red) and suppressor (blue) relative to the complete graph.} 
\label{fig:contourDB}
\end{figure}

 In this case, we have $B\succ A$ if $\gamma>0$ and $x^*< 1/2$, or if $\gamma<0$ and $x^*>1/2$. Similarly, $A\succ B$ if $\gamma>0$ and $x^*> 1/2$, or if $\gamma<0$ and $x^*<1/2$. To see this, from (\ref{jump-prob-complete}), note that the probability of increasing the number of mutants from $k$ to $k+1$ in the complete graph of size $N$ is
\[
p_k^N=\frac{1+\frac{\gamma (x-x^*)}{N}}{2+\frac{\gamma (x-x^*)}{N}},
\]
for $x=k/N$. Therefore, $p_k^N>1/2$ if and only if $\gamma (x-x^*)>0$; $p_k^N<1/2$ if and only if $\gamma (x-x^*)<0$; and $p_k^N=1/2$ if and only if $\gamma (x-x^*)=0$. As result, a positive value of $\psi_1(x)=\gamma (x-x^*)$ benefits type $B$ at state $x$, and a negative value of that benefits type $A$ at $x$. Hence, assuming $x^*<1/2$, and $\gamma>0$ benefits type $B$ for a larger range $[\max \{x^*,0\},1]$, implying $B\succ A$. Similarly, $x^*>1/2$ and $\gamma>0$ benefits type $A$ on a longer interval $[0, \min\{x^*,1\}]$, implying $A\succ B$. The case for $\gamma<0$ can be treated similarly. This can be summarized in the following way:
\begin{itemize}
    \item[i)] $B\succ A$ iff $\gamma(\frac{1}{2}-x^*)>0$;\\[-6pt]
    \item[ii)] $A\succ B$ iff $\gamma(\frac{1}{2}-x^*)<0$;\\[-6pt]
    \item[iii)] $A\sim B$ iff $\gamma(\frac{1}{2}-x^*)=0$.
\end{itemize}
Therefore,
\begin{equation}
\tau=\sign\left(\frac{\mathcal{J}(1/2)}{\mathcal{J}(1)}-\frac{1}{2}\right)=\sign \left(\gamma \left(\frac{1}{2}-x^*\right)\right).
\end{equation}

\noindent One could also directly see this by algebraicly proving that $\mathcal{J}(1/2)/\mathcal{J}(1)-1/2=0$ if and only if $\gamma(1/2-x^*)=0$.

Figure~\ref{fig:contourDB} shows a contour plot of $\beta^\tau$ for $\tau\neq 0$, over the domain $-2 < \gamma < 2$ and $-2 < x^* < 2$, where  $\beta=\lim_{\bar z\rightarrow 0}\phi_{2,\rho}(\bar{z})/\phi_{1}(\bar{z})$ as before. The color scale reflects whether the star graph acts as an amplifier of selection (red, ratio $> 1$), a suppressor (blue, ratio $< 1$), or \aipp (ratio $\approx 1$). The figure illustrates a clear symmetry in how structure affects invasion outcomes. In the case of dominance of $B$ or $A$, that is when either $x^*<0$, or $x^*>2$, the star is a suppressor of selection. 
In the coexistence regime ($\gamma < 0$, $x^* \in (0,1)$), or the coordination regime ($\gamma > 0$, $x^* \in (0,1)$), the star can behave as an amplifier, \aipp, or suppressor, depending on the value of $x^*$. In fact, in these cases the star transitions from suppressor to \aipp to amplifier to \aipp to suppressor as $x^*$ increases. The star is an amplifier of selection when $(\gamma, x)$ is chosen from the area between two curves passing through $(\gamma,x)=(0,1/3)$ and $(\gamma,x)=(0,1/2)$ (these curves remain very close to the lines $x^*=1/3$ and $x^*=1/2$ in this area), and it is a suppressor of selection when the parameters are chosen from out of this closed area.

\subsection{Invasion in the star graph: BD process}\label{ratiobd}

In the BD process, the approximate fixation probability of a single mutant at the center is of order $N^{-2}$, or more precisely it is equal to $N^{-2} f'(0)$. This is much smaller than this probability in the DB process. On the other hand, the approximate fixation probability of a single mutant at a leaf is $f(1/N)$. Therefore, the approximate  invasion probability in the BD process is
\[
\phi_{2,\rho}(\bar z):=\frac{\rho f'(0)}{N^2}+(1-\rho) f(1/N),
\]
where again $\rho=\rho_N$ is the probability that the single mutant occupies the center. Note that
\[
\beta=\lim\limits_{\bar z\rightarrow 0} \frac{\phi_{2,\rho}(\bar z)}{\phi_1(\bar z)}=\lim\limits_{\bar z\rightarrow 0}\frac{\int_0^{\bar z} e^{-\int_0^s (2\psi_1(r)-\psi_2(r))dr}ds}{\mathcal{J}(\bar z)} \cdot \lim\limits_{\bar z\rightarrow 0}\frac{\mathcal{J}(1)}{\int_0^{1-\zbar} e^{-\int_0^s (2\psi_1(r)-\psi_2(r))dr}ds},
\]
where the first limit on the right equals $1$. Therefore,
\[
\beta=\lim\limits_{\bar z\rightarrow 0} \frac{\phi_{2,\rho}(\bar z)}{\phi_1(\bar z)}=\frac{\int_0^{1} e^{-\int_0^s (\psi_1(r)-\psi_2(r))dr}ds}{\int_0^{1} e^{-\int_0^s (2\psi_1(r)-\psi_2(r))dr}ds}.
\]
Note that $\psi_2 - \psi_1 = \psi_2 - 2\psi_1 \quad \text{iff} \quad \psi_1 =0.$
Therefore, if $\psi_1 = 0$ then $\beta = 1$; if $\psi_1 > 0$ then $\beta > 1$, if $\psi_1 < 0$ then $\beta < 1$. On the other hand, we know if $\psi_1 > \psi_2$ then $B \succ A$, if $\psi_1 < \psi_2$ then $A \succ B$, if $\psi_1 = \psi_2$ then $A \sim B$. As a result, the star graph under BD is an amplifier if 
\[
\psi_1 > 0 \geq \psi_2, \quad \text{or} \quad \psi_1 > \psi_2 > 0, \quad \text{or} \quad \psi_2 \geq 0 > \psi_1, \quad \text{or} \quad 0>\psi_2 > \psi_1.
\]
It is a suppressor if $0 > \psi_1 > \psi_2 \quad \text{or} \quad \psi_2 > \psi_1 > 0.
$
It is \aipp if $\psi_1 = 0$.

In particular, letting $\rho=1/N$, for the BD process on the star graph with $\psi_1=r>0$, constant, and $\psi_2=0$, the approximate invasion probability for a uniformly selected mutant site is 
\[
\phi_{2,\rho}(\bar z)=\frac{e^{-2r\zbar}(1-z+e^{2r\zbar}(-1+z-4r\zbar^3))}{-1+e^{2r(-1+\zbar)}}.
\]
Then we have,
\[\beta= \frac{2e^r}{1+e^r},\]
which is an increasing function and for $r>0$, this limit is greater than $1$. Therefore, the star is an amplifier of selection for a sufficiently large population, when the site of the initial single mutant is chosen uniformly at random, and $r>0$.\\

For a general $\rho$, the limit of the invasion probability in the star over the invasion probability in the complete graph is

\begin{equation}\label{rhoSp}
\beta= \frac{2e^r(1-\rho)}{1+e^r},
\end{equation} 
as $\zbar$ tends to zero; see Figure~\ref{inic} for the plot of (\ref{rhoSp}). From Figure~\ref{inic}, we can see that, for a sufficiently large population, if the probability $\rho$ of choosing the center is smaller than $\frac{1}{2}(1-e^{-r})$, the expression in (\ref{rhoSp}) is greater than 1 which means that, for the BD process, the star is an amplifier of selection for a sufficiently large population. On the contrary, if $\rho > \frac{1}{2}(1-e^{-r})$, then the star becomes a suppressor under the BD process. 
For $\rho=\frac{1}{2}(1-e^{-r})$, we have $2e^r(1-\rho)(1+e^r)^{-1}=1$, meaning that the star is \aipp in this case. 

\begin{figure}[htbp]
\begin{center}
\includegraphics[scale=0.52]{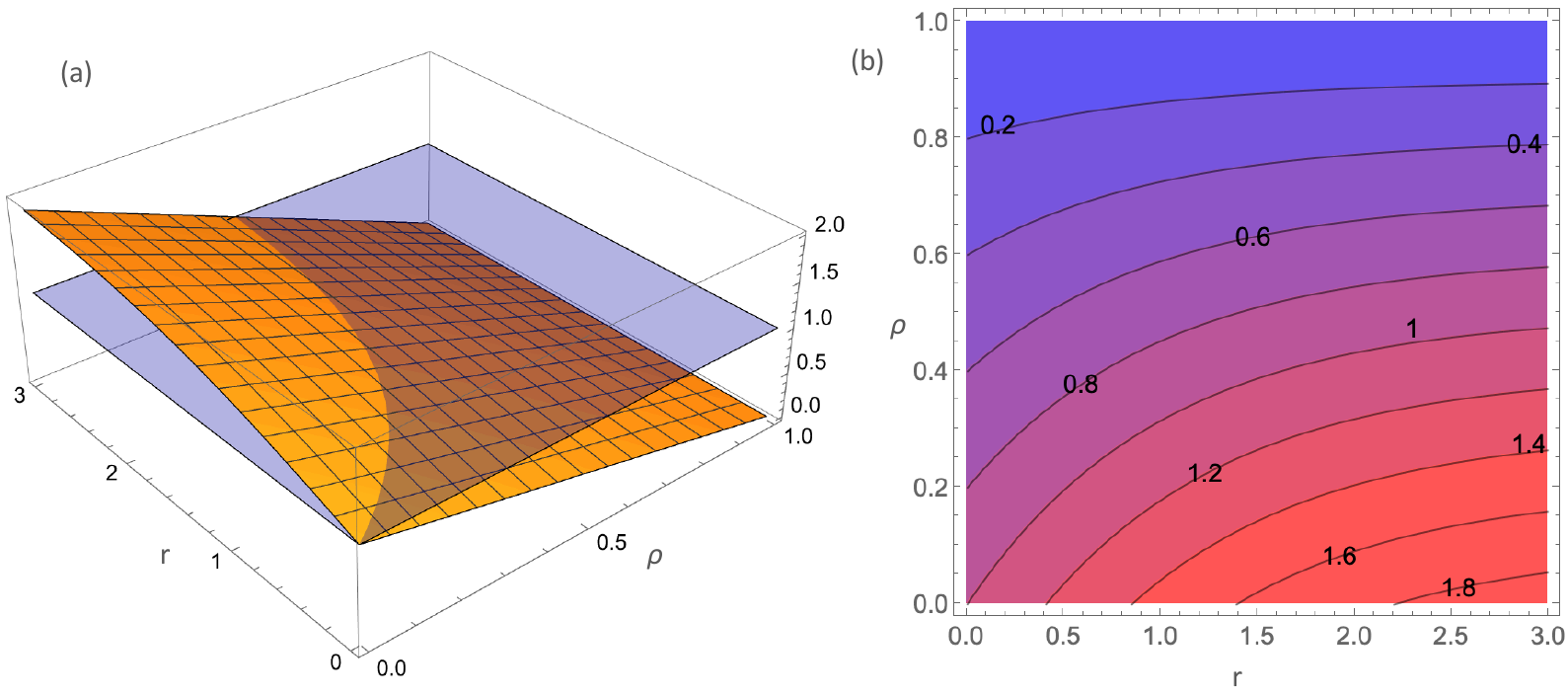}
\end{center}
\vspace*{-5mm}
\caption{Limit of the ratio of invasion probabilities $\beta=\lim_{\bar{z} \to 0}\phi_{2,\rho} / \phi_1$ in the BD process for $\psi_1(x)=r>0$ and  $\psi_2(x) = 0$, over the parameter ranges $0 \leq r \leq 3$ and $0 \leq \rho \leq 1$. (a) 3D surface plot of the limit ratio (\ref{rhoSp}) in orange. (b) Contour plot of the same limit ratio, where the color scale indicates whether the star graph acts as an amplifier (red) or a suppressor (blue) relative to the complete graph.
\label{inic}}
\end{figure}

The discussion above can be summarized as follows. Consider the BD process on the star graph with $N$ vertices. Suppose $\psi_1=r>0$ and $\psi_2=0$. Suppose the center (each leaf, respectively) is selected with probability $0\leq \rho\leq 1$ (with probability $(1-\rho)/(N-1)$, respectively) to be resided by a single mutant in the population:
    \begin{enumerate}
        \item[i)] If $\rho<\frac{1}{2}(1-e^{-r})$, then the star is an amplifier of selection for sufficiently large $N$.\\[-6pt]
        \item[ii)] If $\rho>\frac{1}{2}(1-e^{-r})$, then the star is a suppressor of selection for sufficiently large $N$.\\[-6pt]
        \item[iii)] If $\rho=\frac{1}{2}(1-e^{-r})$, then the star is \aipp.
    \end{enumerate}

In particular, if there is a single mutant in the population that resides at a uniformly random vertex of the star, then, for sufficiently large $N$, the star is an amplifier of selection under the BD mechanism.

We now consider the class of frequency-dependent fitness functions derived from $2 \times 2$ evolutionary games under weak selection, as introduced in Section~\ref{sec:2player}. We consider $\psi_1(x) = \gamma(x - x^*)$, $\psi_2(x) = 0$ and $\rho_N=1/N$, with $-2 < \gamma < 2$ and $-1 < x^* < 2$. This parametrization captures all classical game types: dominance, coexistence, and coordination.

\begin{figure}[h!]
\begin{center}
\includegraphics[scale=0.55]{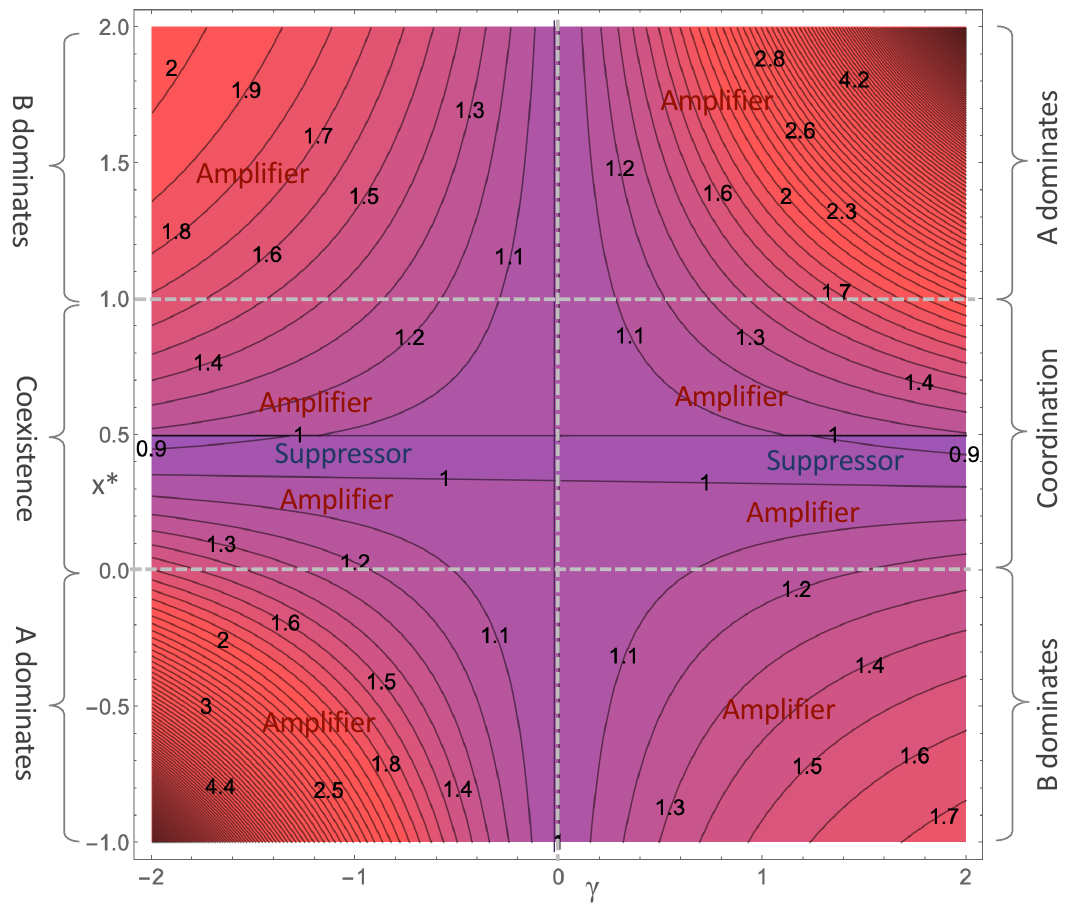}
\end{center}
\vspace*{-5mm}
\caption{
Contour plot of $\beta^\tau$ for the BD process, with fitness functions derived from a $2 \times 2$ evolutionary game under weak selection: $\psi_1(x) = \gamma(x - x^*)$ and $\psi_2(x) = 0$. The domain spans $-2 < \gamma < 2$ and $-1 < x^*< 2$. The color scale indicates whether the star graph acts as an amplifier (red) or suppressor (blue) relative to the complete graph.
\label{fig:contourBD}}
\end{figure}

Figure~\ref{fig:contourBD} presents a contour plot of $\beta^\tau$ for $\tau\neq 0$ in the BD process, over the domain $-2 < \gamma < 2$ and $-1 < x^* < 2$. The color scale indicates whether the star graph acts as an amplifier (red, ratio $> 1$) or suppressor (blue, ratio $< 1$) of selection, or \aipp (ratio $\approx 1$). Interestingly, the qualitative pattern is, in many ways, the reverse of the DB case. While the DB process tends to amplify selection in the coexistence and coordination regimes and suppress it in the dominance regimes, the BD process does the opposite: in both dominance regimes ($x^* < 0$ or $x^* > 2$), the star acts as an amplifier of selection. Meanwhile, in the coordination and coexistence regimes ($x^* \in (0,1)$), the star graph can act as a suppressor, \aipp, or amplifier, depending on $\gamma$ and $x^*$. As in the DB case, the transition between regimes occurs near the lines $x^* = 1/3$ and $x^* = 1/2$, but the amplification and suppression zones are inverted.

\section{Numerical examples}\label{sec:NE}

In this section, we present some numerical examples indicating that our approximations, given in Section~\ref{sec:odecandidate}, for both the BD and the DB processes are quite close to the exact fixation probability vectors, even for a relatively small population size $N$. For each example, we compare the exact fixation probability obtained from solving the equation $\bm{LF} =\bm{F}$ (feasible for not so large $N$) to our continuous approximate fixation probability. 

Our analysis begins with one simple example of a frequency-dependent fitness function, where both the birth and death fitnesses are linear.  Specifically,  in Figure~\ref{graficoN} we consider the BD process with $\psi_1(x) = 2(x-0.5)$ and $\psi_2(x) = x+1$,  which means that the birth fitness function is $\varphi_1(x)=1+2(x-0.5)/N$ and the death fitness function is $\varphi_2(x)=1+(x+1)/N$.

\begin{figure}[h!]
\begin{footnotesize}
\begin{center}
\begin{tabular}{c@{~\!}c}
\includegraphics[scale=0.33]{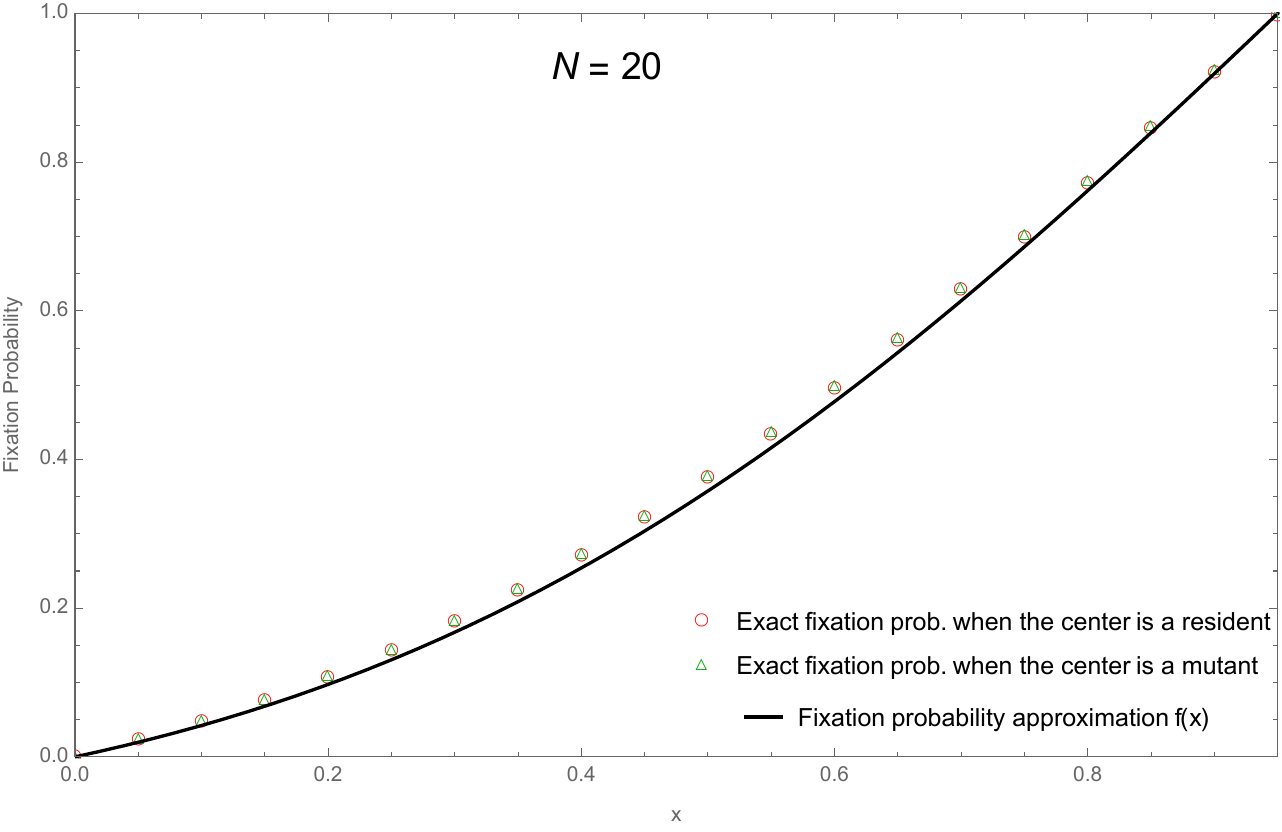}
& \includegraphics[scale=0.33]{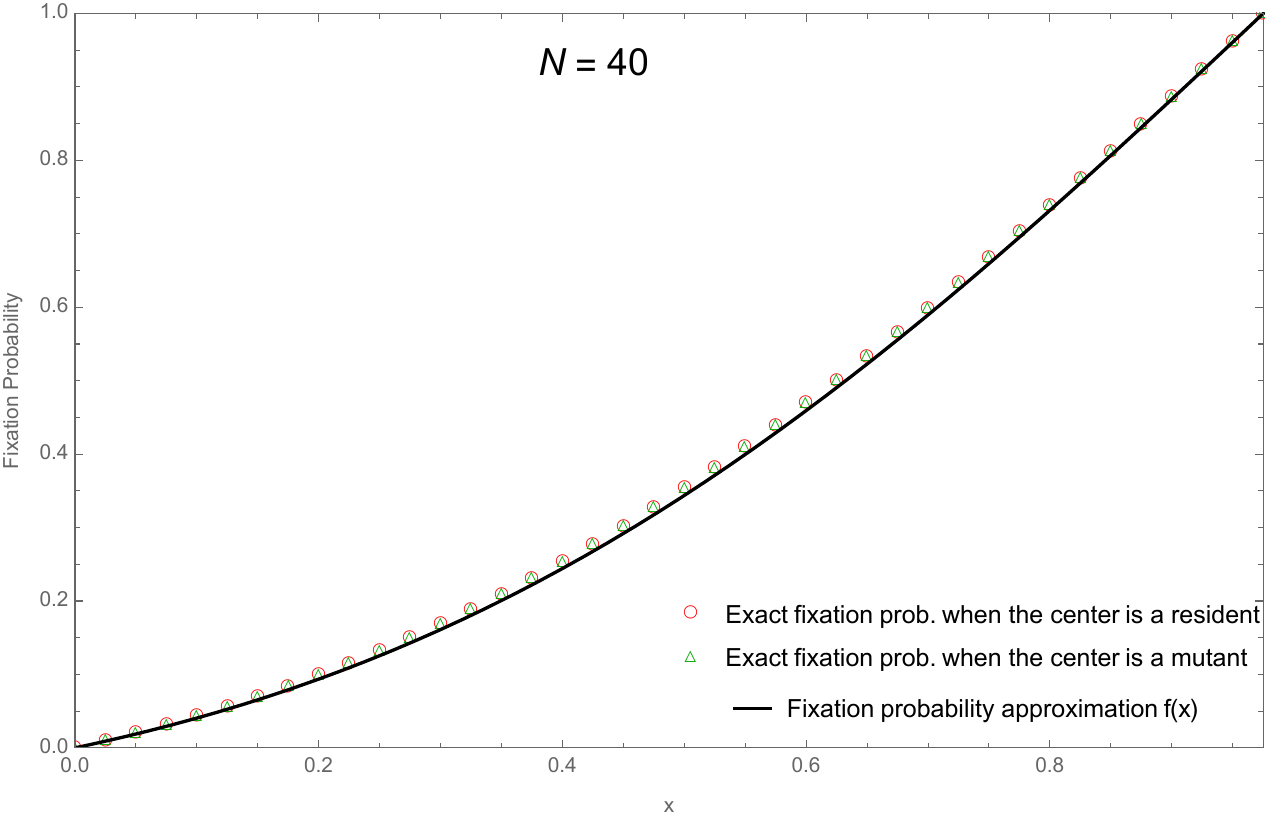}\\
\includegraphics[scale=0.33]{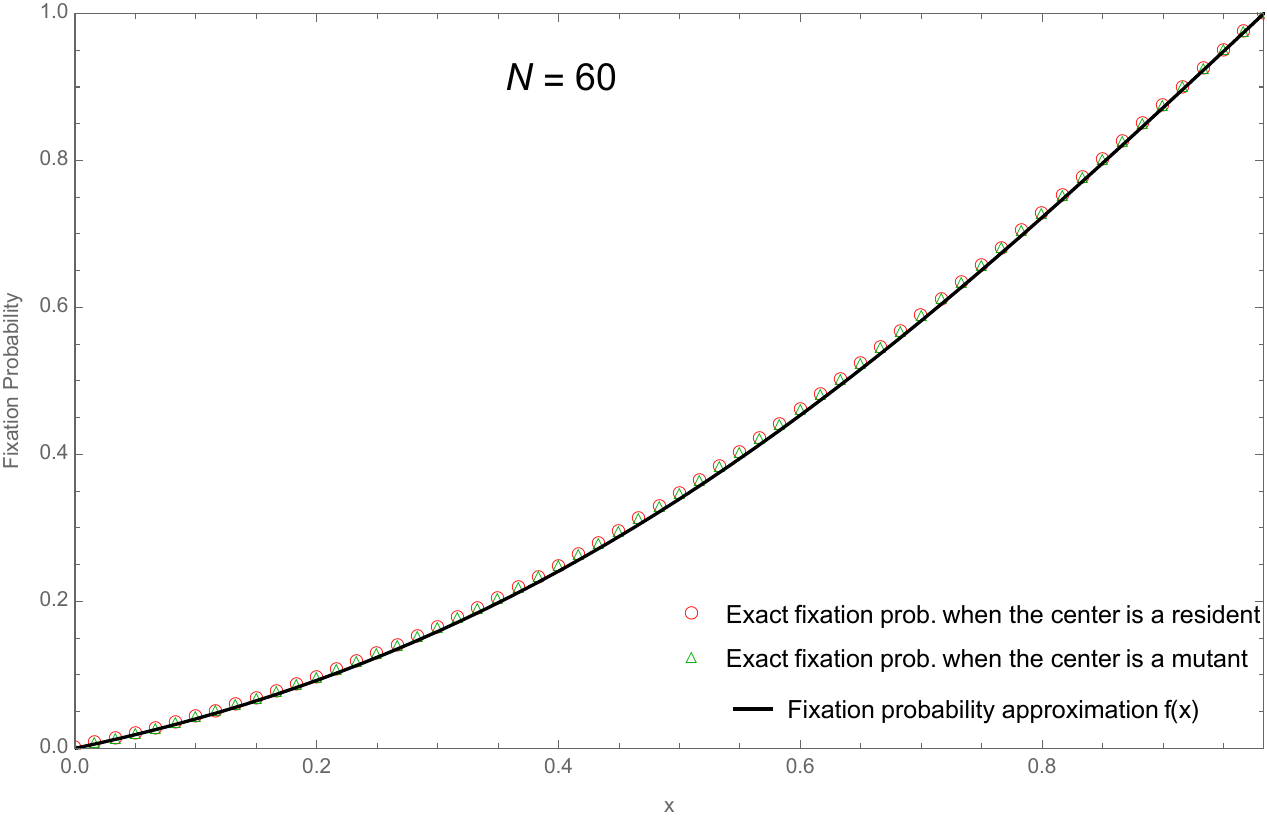}
&\includegraphics[scale=0.33]{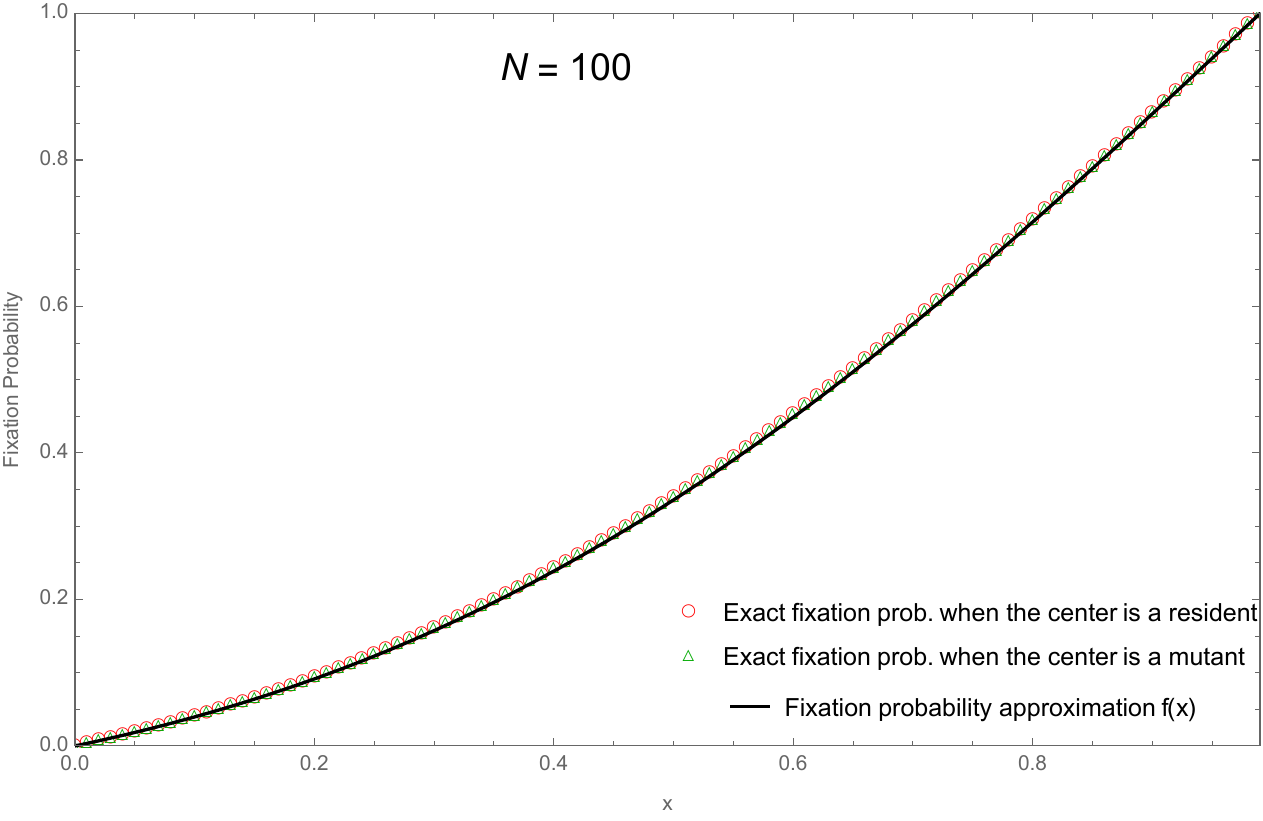}\\
\end{tabular}
\end{center}
\end{footnotesize}
\vspace*{-5mm}
\caption[Examples of our approximate fixation probability vector for the BD process]{The continuous approximation (in black) compared to the exact fixation probability for a population structured as a star graph in the BD process, for $N=20,40,60,100$. The green triangles represent the fixation probability of a population structured as a star with a mutant in the center, and in red circles represent the fixation probability when the center is a resident. In these examples, $\psi_1(x)=2(x-0.5)$ and $\psi_2(x)=x+1$.\label{graficoN}}
\end{figure}

Figure~\ref{graficoN} illustrates this example for the BD process with different population sizes, $N=20, 40, 60,$ and $100$. We can clearly observe that, as $N$ increases, our approximate fixation probability (represented as a solid line) gradually approaches the exact fixation probability (represented by red circles when the center is a resident, and by green triangles when the center is a mutant). To quantitatively assess the accuracy of our approximation, we calculate the error defined as $\Vert \bm{F}^{BD} - \bm{\overline{F}}^{BD}\Vert$, where $\bm{F}^{BD}$ denotes the exact fixation probability vector and $\bm{\overline{F}}^{BD}$ denotes the approximate fixation probability vector for the BD process. Table~\ref{tab1} presents the values of this error for different population sizes. As the population size $N$ grows, we observe a corresponding decrease in the error, confirming the convergence of our approximation.

\begin{table}[ht]
    \centering
    \caption{The fixation probability error for the BD process, i.e. $\Vert \bm{F}^{BD} - \bm{\overline{F}}^{BD}\Vert$ where $\bm{F}^{BD}$ denotes the exact fixation probability vector and $\bm{\overline{F}}^{BD}$ denotes the approximate fixation probability vector for the BD process.}\label{tab1}

    \begin{tabular}{l@{~~~~}c@{~~~~~}c@{~~~~~}c@{~~~~~}c@{~~~~~}c@{~~~~~}c}
        \toprule
        \textit{N} & ~~~20 & 40 & 60 & 80 & 100 & 200 \\
        \cmidrule(lr){1-1} \cmidrule(lr){2-7}
        Error & ~~~0.02036 & 0.01033 & 0.00693 & 0.00522 & 0.00418 & 0.0021 \\
        \bottomrule
    \end{tabular}
\end{table}

In Figure~\ref{DBN}, we compare the fixation probability in the DB process with different population sizes using the same fitness functions as those used in the BD process. Similarly to the BD process, we observe that the error decreases and it is of order smaller than $1/N$ as the population size $N$ increases, see Table~\ref{tab2}.

\vspace*{2mm}
\begin{figure}[h!]
\begin{footnotesize}
\begin{center}
\begin{tabular}{c@{~\!}c}
\includegraphics[scale=0.315]{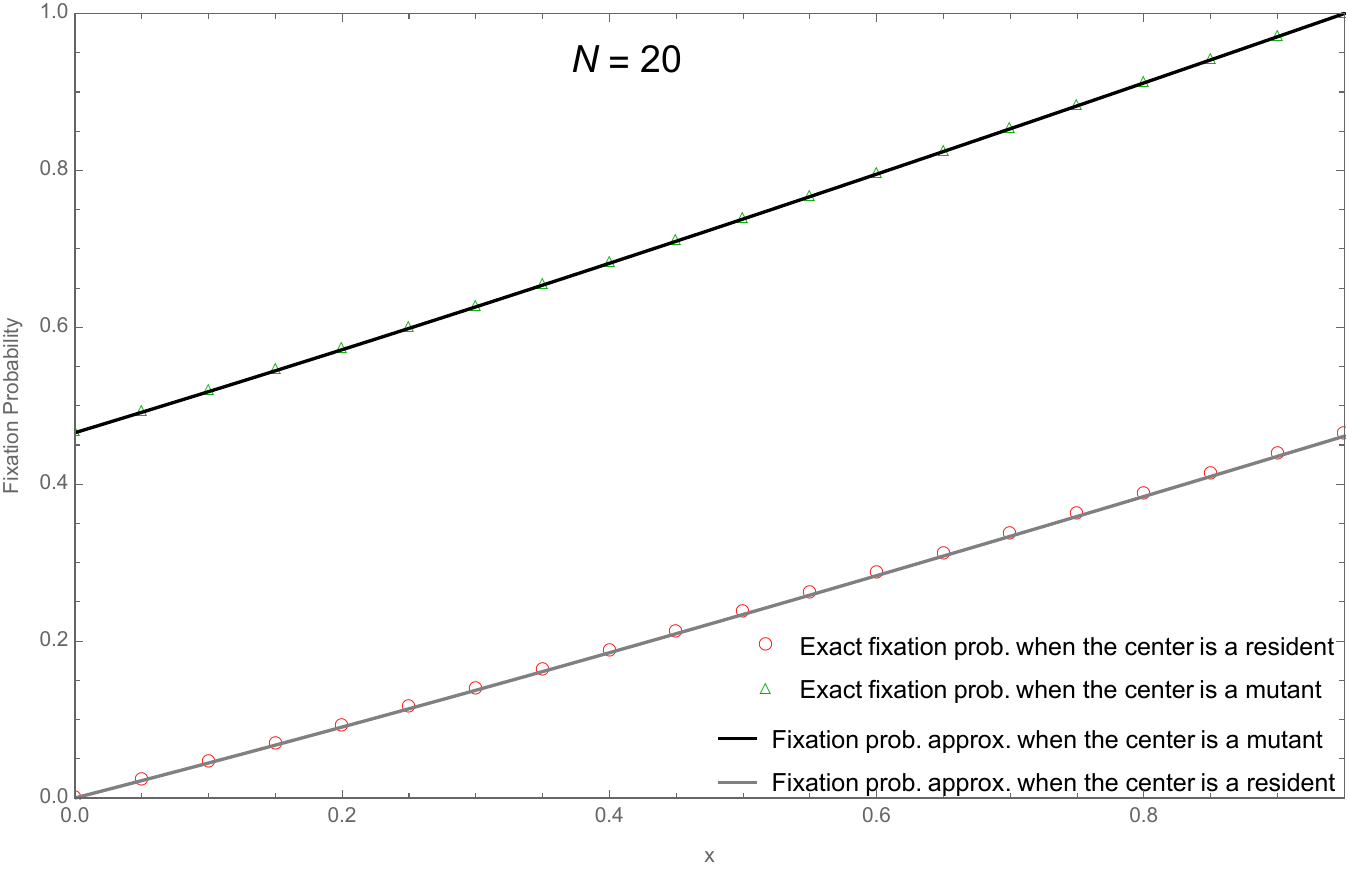}
& \includegraphics[scale=0.315]{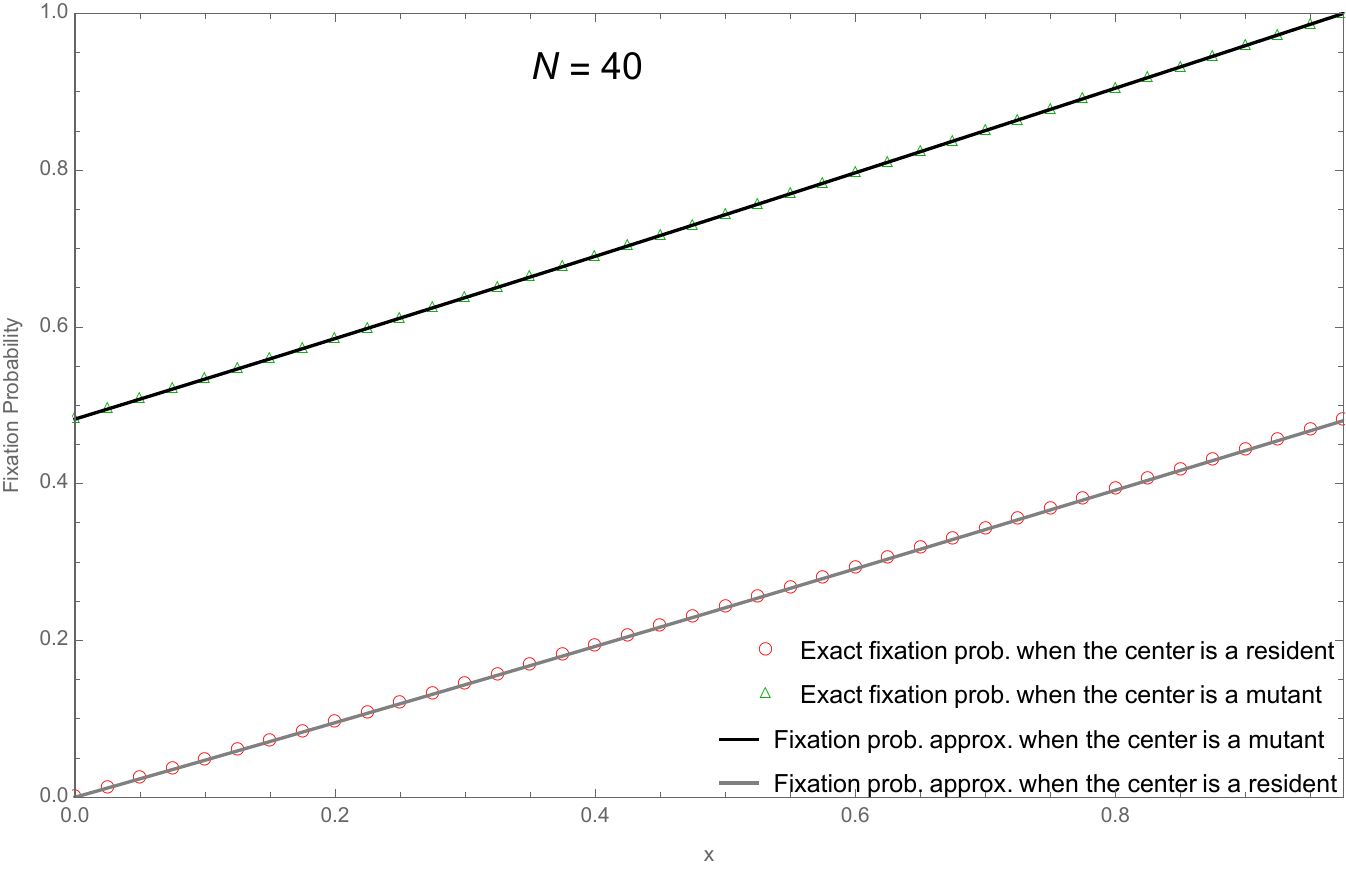}\\
\includegraphics[scale=0.315]{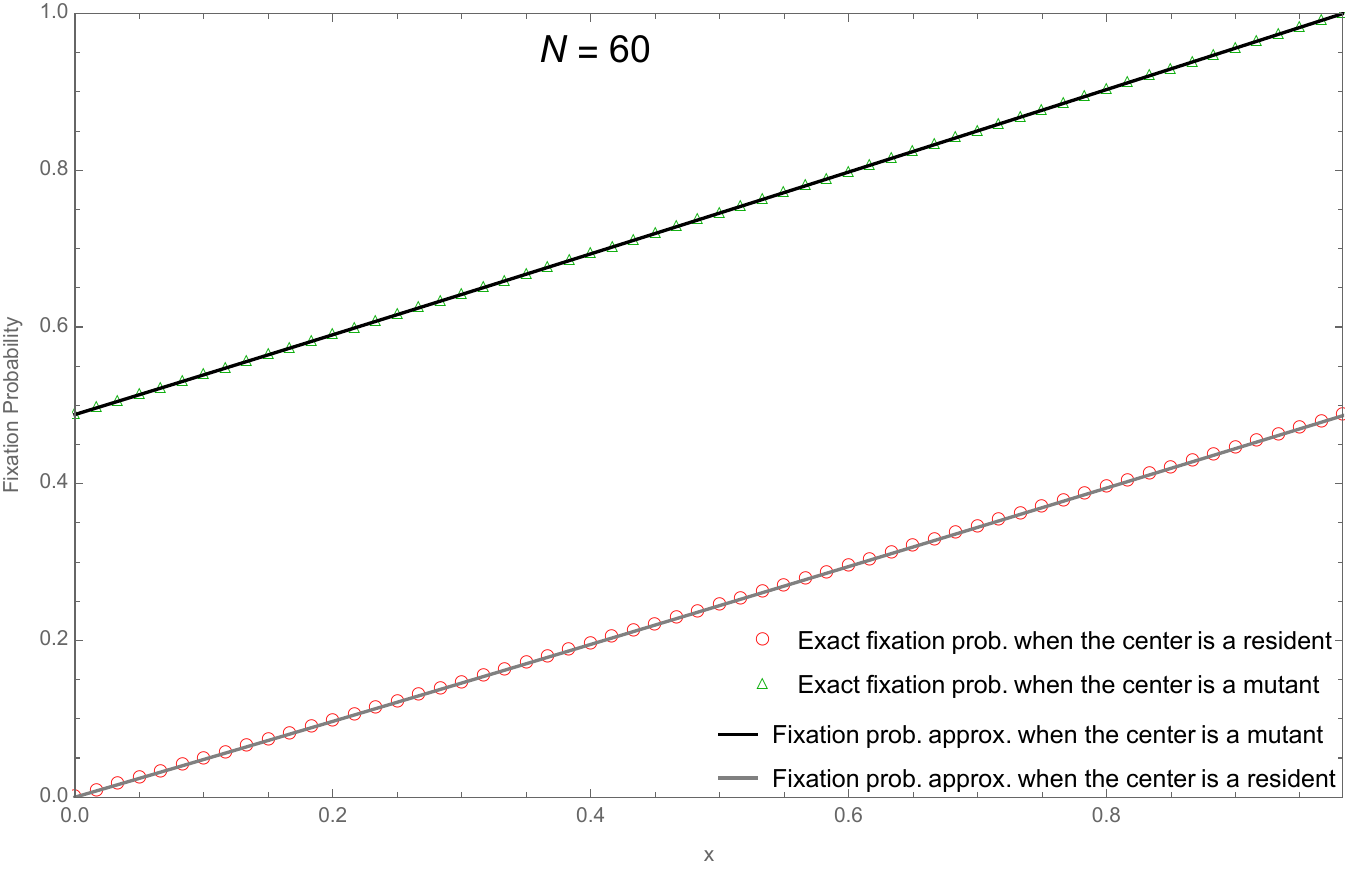}
&\includegraphics[scale=0.315]{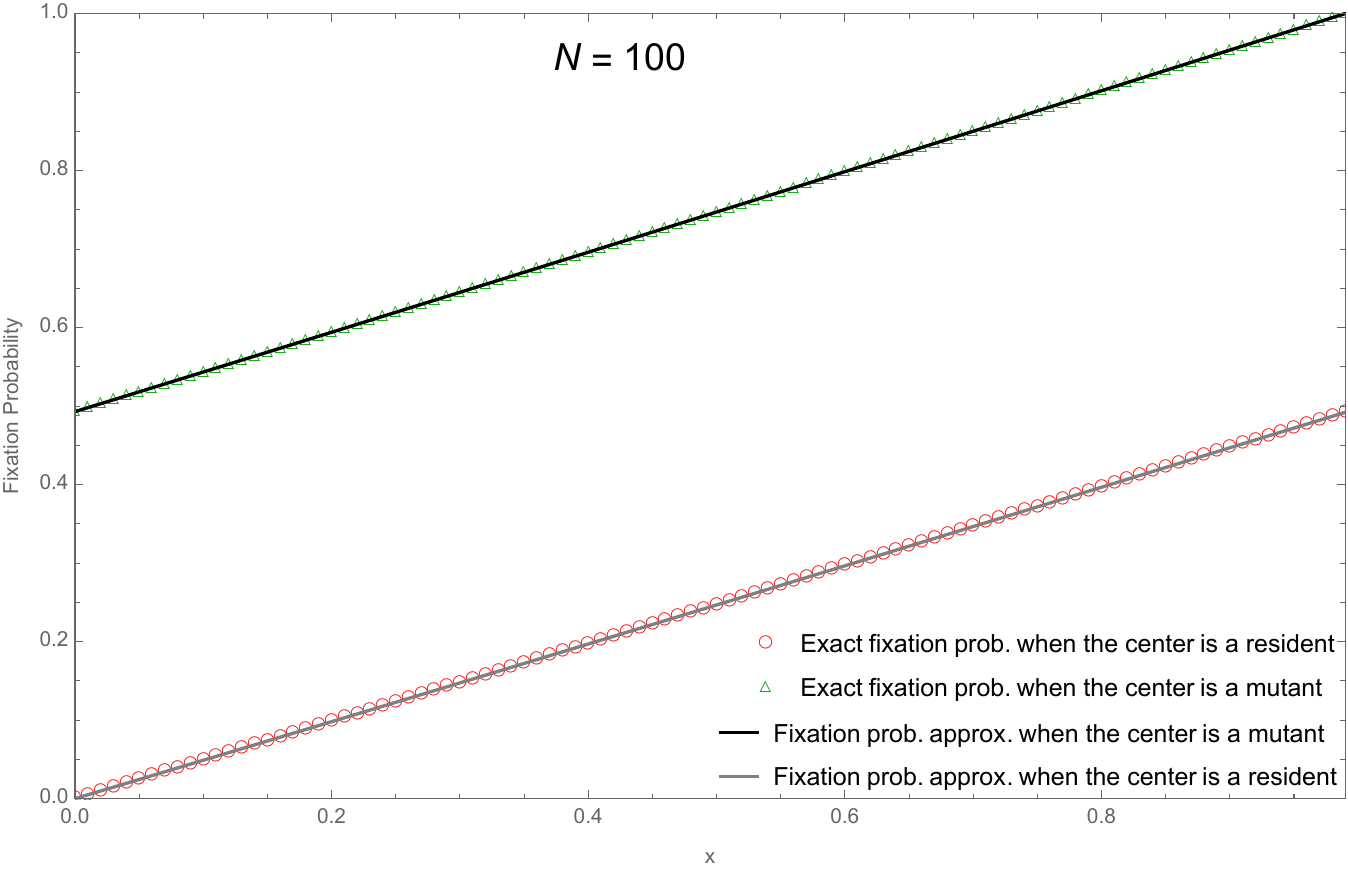}
\end{tabular}
\end{center}
\end{footnotesize}
\vspace*{-5mm}
\caption[Examples of our approximate fixation probability vector for the DB process]{The continuous approximation (in black) compared to the exact fixation probability (in green triangles) for a population structured as a star graph when the center is a mutant, in the DB process. When the center is a resident,  the continuous approximation and the exact fixation probability are given in gray and red circles, respectively. The fitness functions are $\psi_1(x)=2(x-0.5)$ and $\psi_2(x)=x+1$, for $N=20,40,60,100$. \label{DBN}}
\end{figure}

\begin{table}[ht]
    \centering
    \caption{The fixation probability error for the DB process, i.e. $\Vert \bm{F}^{DB} - \bm{\overline{F}}^{DB}\Vert$ where $\bm{F}^{DB}$ denotes the exact fixation probability vector and $\bm{\overline{F}}^{DB}$ denotes the approximate fixation probability vector for the DB process.}\label{tab2}%

    \begin{tabular}{l@{~~~~}c@{~~~~~}c@{~~~~~}c@{~~~~~}c@{~~~~~}c@{~~~~~}c}
        \toprule
        $N$  &~~~20 & 40 & 60 &80 &100 &200\\
        \cmidrule(lr){1-1}\cmidrule(lr){2-7}
        Error &~~~0.00284 &0.00073 &0.00033 &0.00019 &0.00012 &0.00003\\
        \bottomrule
    \end{tabular}
\end{table}

Furthermore, the comparison of Figures~\ref{graficoN} and \ref{DBN} highlights a clear distinction between the mechanisms of the BD and DB processes.
In the DB process, the presence of a mutant or resident at the center leads to a drastic difference. When a mutant starts from the center, the invasion probability is close to $1/2$, while an invasion starting from a leaf diminishes significantly and approaches zero.

In contrast to previous results in the literature, which just focused on linear or constant fitness functions, our study extends its scope to include a wider range of frequency-dependent fitness functions. For instance,in the BD process, we compare four distinct fitness function types: constant, linear, polynomial, and Gaussian, as shown in Figure~\ref{graficoBD}.

\vspace*{2mm}
\begin{figure}[h!]
\begin{center}
\begin{tabular}{c@{\!~}c}
\includegraphics[scale=0.373]{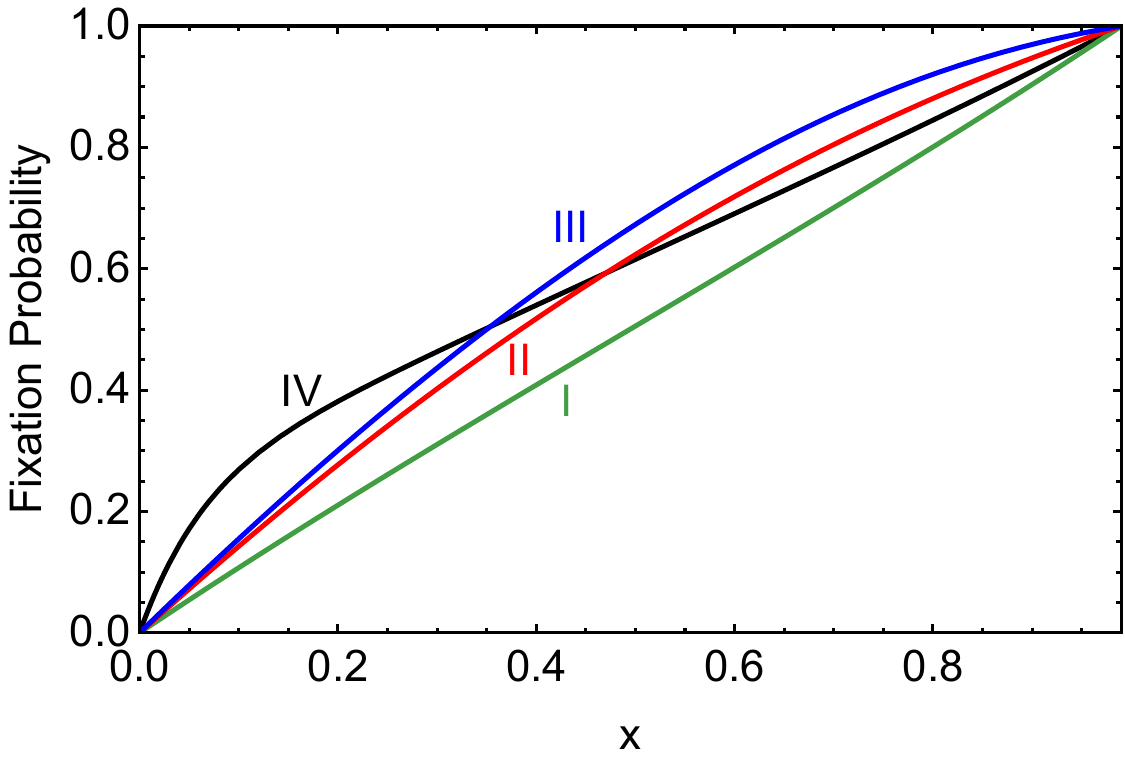}
&\includegraphics[scale=0.373]{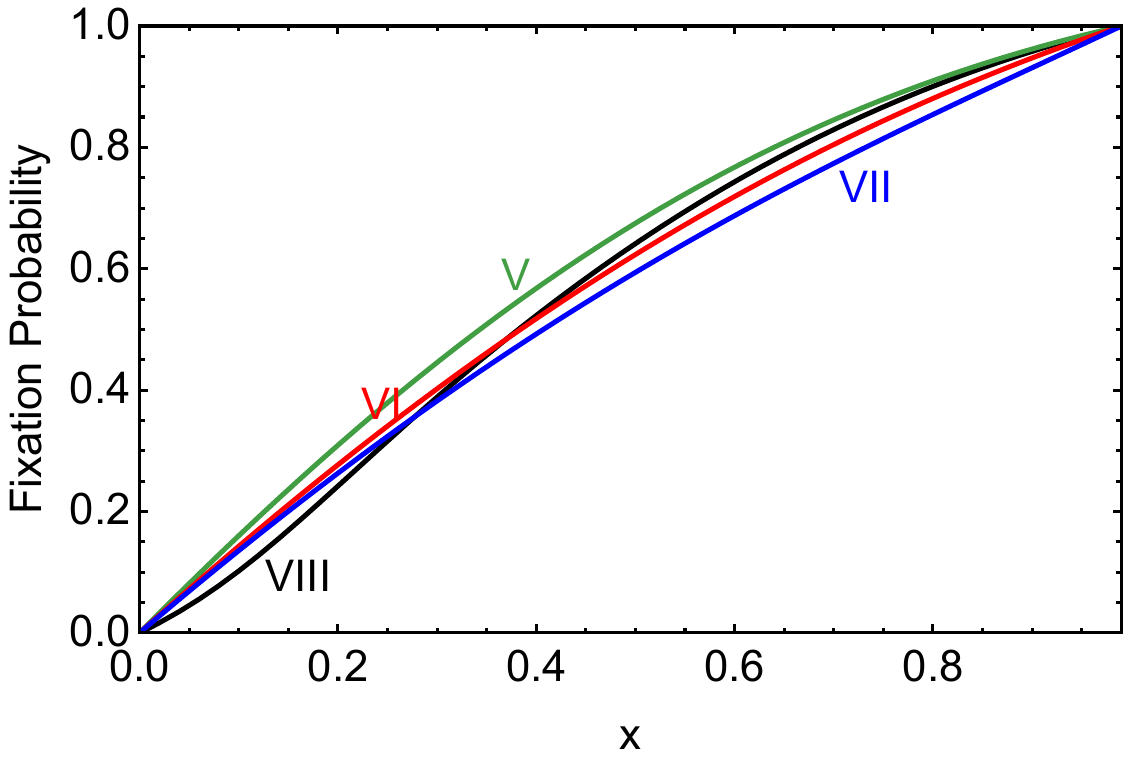}
\end{tabular}
\end{center}
\vspace*{-5mm}
\caption[Comparing different fitness functions in the BD process]{Examples of the approximate fixation probability for various fitness functions in the BD process. On the left, we have $\psi_2(x)=0.5+x$ and four different birth fitnesses: (I) constant $\psi_1(x)=0.5$; (II) linear $\psi_1(x)=0.5+x$; (III) polynomial $\psi_1(x)=0.5+x+x^3+x^5$; and (IV) Gaussian $\psi_1(x)=0.5+6 e^{-40x^2}$. On the right, we have $\psi_1(x)=0.5+x$ and four different death fitnesses: (V) constant $\psi_2(x)=0.5$; (VI) linear $\psi_2(x)=0.5+x$; (VII) polynomial $\psi_2(x)=0.5+x+x^3+x^5$; and (VIII)  Gaussian $\psi_2(x)=0.5+6 e^{-40x^2}$. In all of these scenarios, the population size is $N=100$.
\label{graficoBD}}
\end{figure}

On the left-hand side, we examine the approximate fixation probability across varying birth fitness functions: (I) constant $\psi_1(x)=0.5$; (II) linear $\psi_1(x)=0.5+x$; (III) polynomial $\psi_1(x)=0.5+x+x^3+x^5$; and  Gaussian $\psi_1(x)=0.5+6 e^{-40x^2}$. In each of these scenarios, $\psi_2(x)=0.5+x$. Conversely, on the right-hand side, we shift our focus, maintaining birth fitness as $\psi_1(x)=0.5+x$, while manipulating death fitness as follows: (V) constant $\psi_2(x)=0.5$; (VI) linear $\psi_2(x)=0.5+x$; (VII) polynomial $\psi_2(x)=0.5+x+x^3+x^5$; and (VIII) Gaussian $\psi_2(x)=0.5+6 e^{-40x^2}$. 

Note that the Gaussian fitness function is greater than the polynomial fitness function for $x\in(0,u)$, where $u\approx 0.27$, and eventually falls below the polynomial fitness for $x\in(u,1]$. On the other hand, the fixation probability is a function of fitness, such that the higher the birth fitness at a state $x$, the greater the chance of fixation starting at $x$. Therefore, we anticipate that the difference in fixation probabilities under Gaussian fitness and polynomial fitness changes sign from positive to negative. It is important to observe that this change of sign does not occur exactly at $u$, but rather at $u'\geq u$, which can be approximated as $u'\approx 0.35$ when comparing (III) and (IV) cases. With this perspective in view, we also expect the fixation probability in case (III) to always exceed the fixation probability in case (II), and the latter to always exceed that in case (I) when starting at the same state $x$. 

On the contrary, in the BD process, a higher death fitness at a state $x$ leads to  a smaller chance of fixation initiating at $x$. Consequently, on the right-hand side, the situation is reversed. More precisely, in case (V), the fixation probability is greater than that in case (VI), and the latter is greater than the probability in case (VII) when starting from the same state $x$. On the other hand, a sign change for the difference of fixation probabilities under Gaussian and polynomial death fitnesses occurs at approximately $u'\approx 0.28$, as observed when comparing (VII) and (VIII). 

The comparison between the left and right sides of Figure~\ref{graficoBD} serves to emphasize the points discussed in Section~\ref{sec:equivalence}. As mentioned, one cannot approximate the fixation probability of $A$ by switching the birth fitness of $B$ and the death fitness of $A$ and vice versa as seen in the case of the complete graph. We also note that, the difference of the fixation probabilities for any two cases on the left-hand side is greater than the difference of fixation probabilities for any two analogous cases on the right-hand side.

We now present some similar examples for the DB process in Figure~\ref{graficoDB}. On the left-hand side, we explore the approximate fixation probability for different birth fitness functions: (I) constant $\psi_1(x)=5$; (II) linear $\psi_1(x)=5+10x$; (III) polynomial $\psi_1(x)=5+10x+10x^3+10x^5$; and (IV) Gaussian $\psi_1(x)=5+60 e^{-40x^2}$. In each of these cases, $\psi_2(x)=5+10x$. On the other hand, on the right-hand side, our focus shifts to maintaining birth fitness as $\psi_1(x)=5+10x$, while adjusting death fitness as follows: (V) constant $\psi_2(x)=5$; (VI) linear $\psi_2(x)= 5+10x$; (VII) polynomial $\psi_2(x)=5+10x+10x^3+10x^5$; and (VIII) Gaussian $\psi_2(x)=5+60 e^{-40x^2}$. 

\vspace*{2mm}
\begin{figure}[h!]
\begin{center}
\begin{tabular}{c@{\!~~}c}
\includegraphics[scale=0.235]{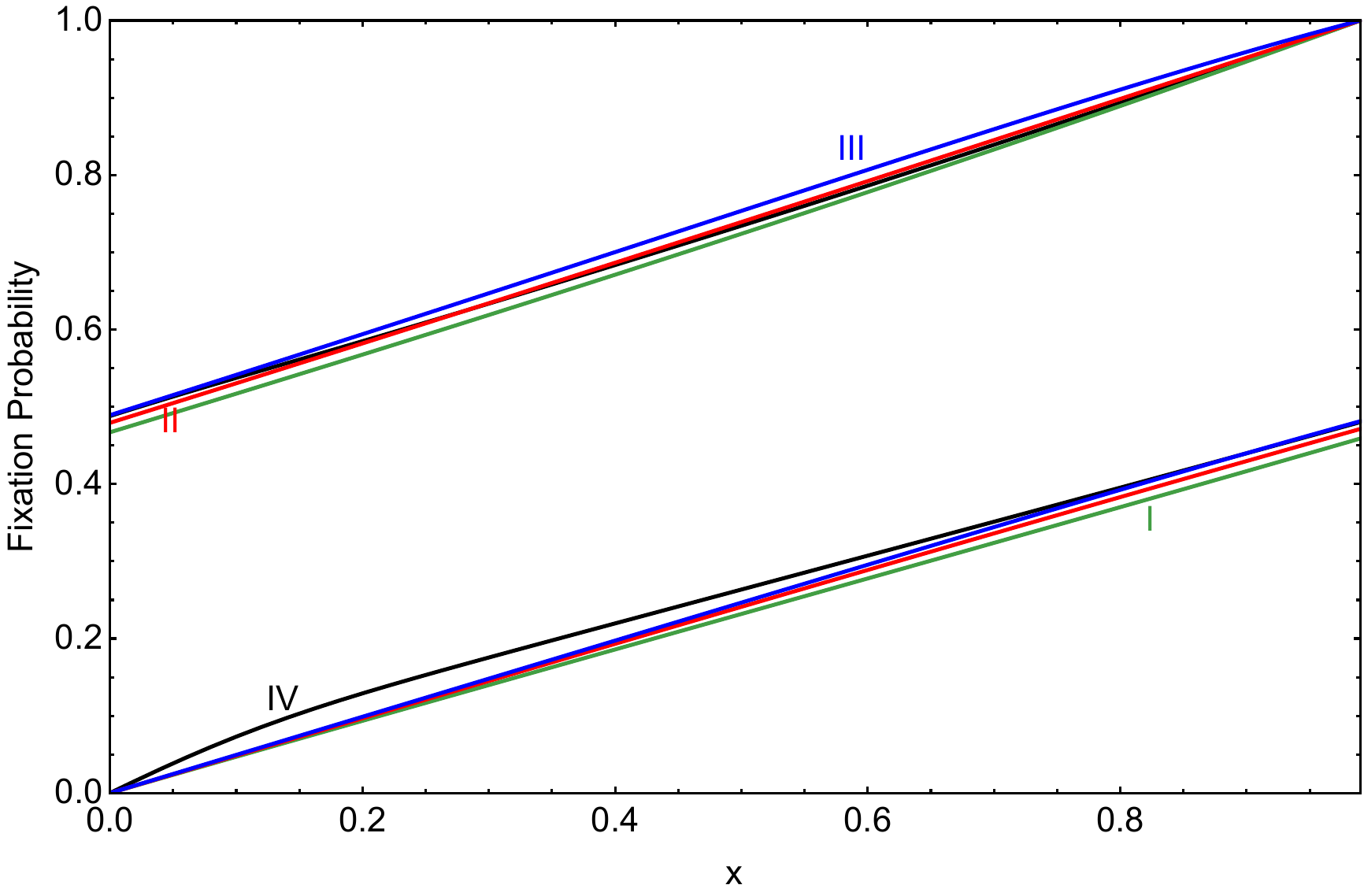}
&\includegraphics[scale=0.235]{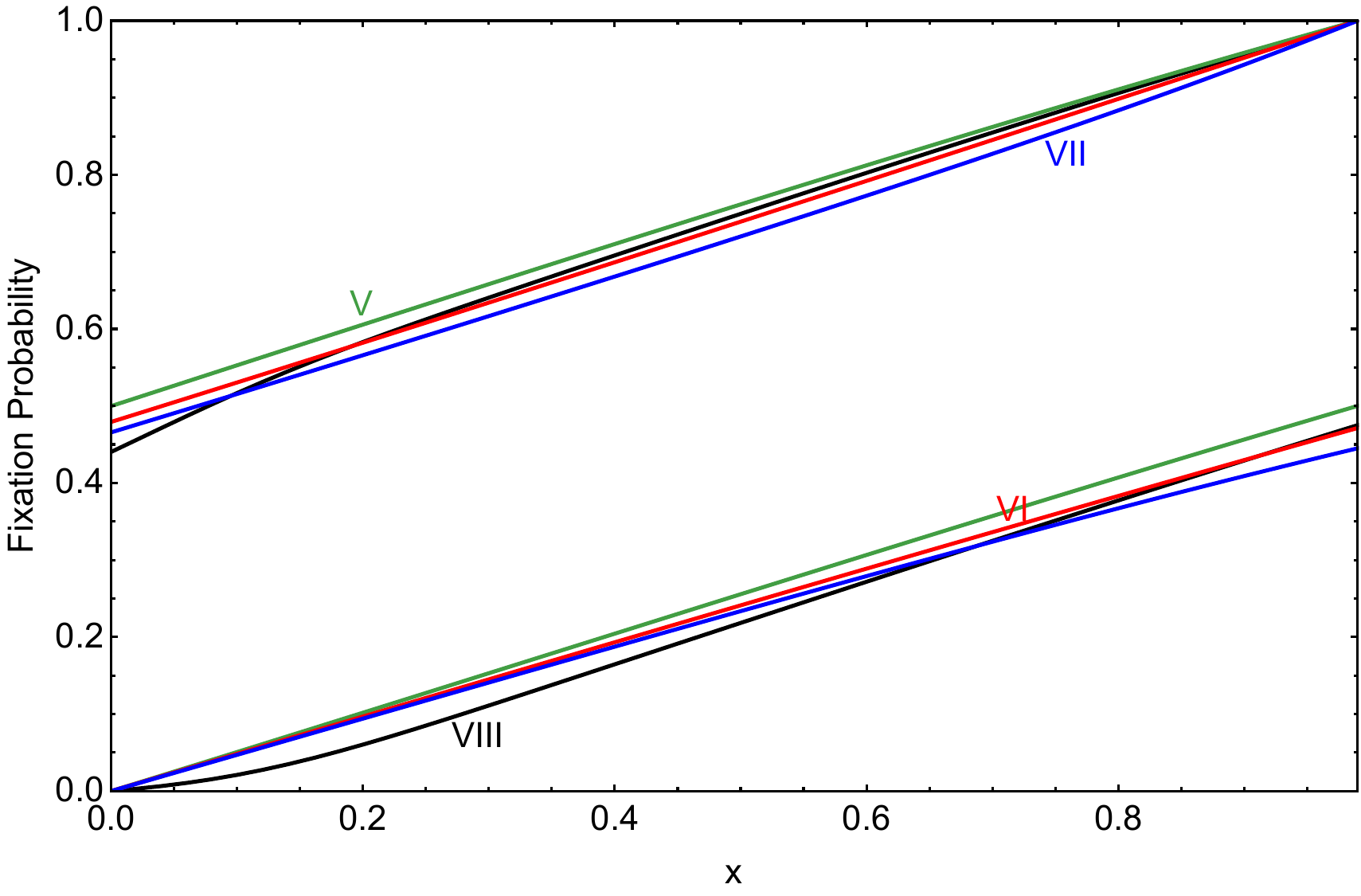}
\end{tabular}
\end{center}
\vspace*{-5mm}
\caption[Comparing different fitness functions in the DB process]{Examples of the approximate fixation probability for various fitness functions in the DB process. On the left, we have $\psi_2(x)=5+10x$ and four types of birth fitness: (I) constant $\psi_1(x)=5$; (II) linear $\psi_1(x)=5+10x$; (III) polynomial $\psi_1(x)=5+10x+10x^3+10x^5$; and (IV) Gaussian $\psi_1(x)=5+60 e^{-40x^2}$. On the right, we have $\psi_1(x)=5+10x$ and four types of death fitness: (V) constant $\psi_2(x)=5$; (VI) linear $\psi_2(x)= 5+10x$; (VII) polynomial $\psi_2(x)=5+10x+10x^3+10x^5$; and (VIII) Gaussian $\psi_2(x)=5+60 e^{-40x^2}$. In all of these scenarios, the population size is $N=100$.
\label{graficoDB}}
\end{figure}

As in the BD case, the fixation probability for the DB process depends on its fitness. More specifically, a higher birth fitness at a state $x$ corresponds to an increased chance of fixation starting at $x$, while a higher death fitness at $x$ leads to a smaller chance of fixation starting at that state. Therefore, as in the BD process, we expect the fixation probability in scenario (III) consistently exceeds that in scenario (II), and the latter always surpasses the probability in scenario (I) when initiating from the same state $x$. Furthermore, we expect that the fixation probability in scenario (V) will always be greater than that in scenario (VI), with the latter being consistently higher than the probability in scenario (VII) under the same starting state $x$.

 This extended analysis broadens our understanding beyond the limitations of linear or constant fitness functions prevalent in prior research, thereby shedding light on the dynamics of more complicated frequency-dependent fitness functions.

\section{Discussion}

In contrast to previous results in the literature, which only consider constant or linear fitness functions (cf. \cite{nowak1,br2008,Houchmandzadeh,monk2014martingales, kkk, hindersin2015most}), this paper provides continuous approximations for the fixation probability of large populations living on the star graph, considering general frequency-dependent (smooth) birth and death fitness functions under the weak-selection regime.  As an application, we compare various types of fitness functions, including constant, linear, polynomial, and Gaussian functions, as presented in Section~\ref{sec:NE}. The success of our methodology relies on the assumption of a large but finite population and the high symmetry of the population structure, such as that of the star graph. Due to the symmetry and automorphism of the star graph, the dynamics and fixation probability can be described by a system of recursive equations, with the number of equations being twice the population size. As we seek a continuous approximation, the analysis reduces to a system of two ordinary differential equations. In general, this method can be extended to populations with other structures, as long as the graph representing the interactions of the population is relatively simple, allowing the fixation probability to be expressed as a system of a relatively small number of recursive equations.  For more on the number of mutant-resident configurations, see ~\cite{br2008}. Of course, the technical details for determining the order of error will vary for each specific graph and update mechanism. However, we expect that it will be easy to verify through numerical examples that the error of the approximation is small.

As we have seen in this study, the details for determining the order of the error of the approximation differ between the BD and the DB processes. In the DB case, we proved that the approximation error is of order $1/N$, where $N$ represents the size of the population. That is, the larger the population, the smaller the error. In the BD case, we found a very similar approximate fixation probability as in the complete graph presented in~\cite{chalubsouza16}.  In both updating mechanisms, throughout various numerical examples, we illustrated that our approximations are quite close to the exact fixation probability even for populations that are not very large.

While the methodology used to derive our continuous approximation of the fixation probability is similar for both the BD and DB processes, the resulting approximations exhibit distinct forms. This fact is due to fundamental differences in how the two update rules interact with the star graph structure. The star can be viewed as a structured population in which the $N-1$ leaves only interact indirectly through the central node. In the BD process, the center has the same individual reproduction probability (depending on fitness) as any single leaf, making it much less likely to be chosen for reproduction than the leaves collectively. As a result, it often changes state numuruse times before the number of mutants among the leaves increases or decreases. This creates a dynamic where changes in mutant frequency are driven by rare, coordinated events. In contrast, in the DB process, the death event occurs first, and the replacing individual is chosen proportionally to fitness. As a result, the type occupying the center can repeatedly convert neighboring leaves before being selected for death. This creates a process where the central node exerts strong directional influence on the leaf population, and fixation dynamics are driven by asymmetric transitions favoring the type in the center. Because of this directional asymmetry, the effective dynamics meaningfully differ between BD and DB, and the respective approximations must be tailored to capture these distinct behaviors.

In this paper, we only consider the simplest case where we have undirected connections between two vertices. Note that a straightforward generalization is to consider a directed star graph in which all arcs heading to one of the leaves have the same weight, say $w_1$, and all arcs heading to the center of the star have another weight, say $w_2$.  
It is evident that we can approach this scenario similarly and achieve similar results to those presented in this paper. However, if we assign distinct weights to each arc of the star, despite its simple structure, we would forfeit the advantage of the star graph's symmetry. This is because, to compute the fixation probability, we must have knowledge of all the weights, preventing us from simplifying the analysis into a system of recursive equations where the number of equations is merely twice the population size. 

As mentioned above, our fixation probability approximation for the BD process, when the population resides on a star graph, is very similar to the approximation provided by~\cite{chalubsouza16} for populations living on complete graphs. Consequently, we can infer that the asymptotic qualitative behavior of a population structured as a star graph mirrors that of a population structured as a complete graph. This comparison is highlighted in Section~\ref{sec:2player}, especially when the fitness is determined by a linear function derived from a two-player game under weak-selection conditions. Furthermore, we specifically investigated whether the star graph acts as an amplifier or a suppressor of selection, providing a detailed analysis of invasion probabilities in Section~5. This section includes a threshold-based criterion for distinguishing between amplification and suppression regimes, depending on selection intensity and the initial position of the mutant. While we focused on fixation probabilities throughout the paper, our analytical framework can also be adapted to study fixation times. Extending our methodology in that direction could offer new insights into the expected time to fixation in structured populations.

Beyond its traditional applications in population genetics, the Moran process on graphs, which includes both birth-death and death-birth update mechanisms, plays an essential role in comprehending the dynamics of cancer evolution. By modeling the evolutionary trajectories and interactions between normal and tumor cell populations within spatially constrained environments, the Moran process on graphs offers valuable insights into tumor growth dynamics
~\citep{kkk,nanda2017,coggan2022,broom2022,wolfl2022}. These insights are crucial for understanding the dynamics of stem cells and the significance of structures capable of either suppressing somatic evolution leading to cancer or amplifying selection pressures~\citep{nowak1,Hindersin16,Tkadlec21,Traulsen22}. Furthermore, in this paper, we enhanced the existing models in the literature by offering new approximate fixation probabilities for a population residing on a star graph, considering general birth and death fitness functions. This enhancement provides greater flexibility in selecting birth and death fitness parameters, thereby improving the model's fit when applied to real data analyses. 

\vspace*{0.5cm}
\paragraph{Acknowledgments}
This paper was completed while the first author was affiliated with Columbia University. We thank Professors Mark Broom, Fabio A.C.C. Chalub, Jo\~{a}o Meidanis, Armando G. Neves and Ralph C. Teixeira for their comments on an earlier version of the work.
PHdS was supported in part by FAPERJ grant E-26/200.258/2015. MOS thanks the support by CAPES - Brazil - Finance code 001 and the support by FAPERJ-Brazil through the grant E-26/210.440.2019.

\begin{appendices}

\section{Proof of Theorem~\ref{rDB}}\label{secA1}

This section is devoted to the proof of Theorem~\ref{rDB}.  First, note that we can rewrite the matrix $\bm{L}$ as follows
\[
\bm{L}=\left[ \begin{array}{ccc}
1&\textbf{0}^t &0\\
\bm{\alpha}&\bm{\widetilde{L} }&\bm{\beta}\\
0&\textbf{0}^t&1
\end{array}\right] 
\]
where $\textbf{0}$ is the null vector of dimension $2N-2$, $\bm{\widetilde{L}}$ is a matrix with dimension $2N-2\times 2N-2$, $\bm{\alpha}$ is a vector of dimension $2N-2$ such that $\alpha_1=c_{\zbar}$, $\alpha_{N}=a_{0}$ and $\alpha_i=0$ for all $i$ different than  $1$ and $N$, and $\bm{\beta}$ is a vector of dimension $2N-2$ such that $\beta_{N-1}=d_{(N-1)\zbar}$, $\beta_{2N-2}=b_{(N-2)\zbar}$ and $\beta_i=0$ for all $i$ other than $N-1$ and $2N-2$.  The following result from~\cite{chalubsouza} shows that there exists a unique fixation probability vector $\bm{F}^t=[0~~\bm{\widetilde{F}}^t~~1]^t$, where $\bm{\widetilde{F}}$ is a vector of dimension $2N-2$, since $p^N_{2,0}=0$ and $p^N_{1,1- \zbar}=1$. 

\begin{proposition}[\cite{chalubsouza}]\label{fixation}
Let $\bm{\widetilde{M}}=\bm{\widetilde{L}}-\bm{\widetilde{I}}$. Then, there exists a unique vector $\bm{\widetilde{F}}\in \rr^{N-2}$, with $0<(\widetilde{F})_i<1$, such that $\bm{F}^t=[0~~\bm{\widetilde{F}}^t~~1]^t$, with $\bm{M F}=\bm{0} $. It satisfies 
\[
\bm{\widetilde{F}}=-\bm{\widetilde{M}}^{-1}\bm{\beta}.
\]
\end{proposition}

So we are reduced to estimate $\bm{\widetilde{F}}$. To this end, let $\widetilde{\overline{F}}^{DB}_{i-1}=\overline{F}^{DB}_i$ for $i=2,\dots,2N-1$  as defined in Section~\ref{odeDB}. Using Proposition~\ref{fixation}, we get
\begin{equation}\label{erro}
\Vert\bm{\widetilde{F}} - \bm{\widetilde{\overline{F}}} \Vert=\Vert-\bm{\widetilde{M}}^{-1}\bm{\beta} - \bm{\widetilde{\overline{F}}} \Vert=\Vert\bm{\widetilde{M}}^{-1}(-\bm{\beta} - \bm{\widetilde{M}}\bm{\widetilde{\overline{F}}}) \Vert\leq \Vert\bm{\widetilde{M}}^{-1}\Vert\Vert \bm{\beta}+\bm{\widetilde{M}}\bm{\widetilde{\overline{F}}}\Vert.
\end{equation}

Hence, in order to estimate the error $\Vert\bm{\widetilde{F} }- \bm{\widetilde{\overline{F}}} \Vert$, we need appropriate upper bounds for $\Vert\bm{\widetilde{M}}^{-1}\Vert$ and $\Vert \bm{\beta}+\bm{\widetilde{M}\widetilde{\overline{F}}}\Vert$.  This is done in Proposition~\ref{cotaM} and Proposition~\ref{MF}. 
To see the details for the DB process, we write $\bm{\widetilde{M}}=\zbar\bm{\widetilde{M}}_0+\zbar^2\bm{\widetilde{M}}_1$, where $\bm{\widetilde{M}}_0$ is the matrix defined by
\begin{equation}\label{M0}
\left(\widetilde{M}_0\right)_{ij}=\left\lbrace \begin{array}{ll}
-1/\zbar, & \text{ if }  i=j; \\
-1+1/\zbar, &  \text{ if } j= i-1 \text{ and } 2\leq i\leq N-1; \\
1, &  \text{ if } j= N+i  \text{ and } 1\leq i\leq N-2; \\
1, &  \text{ if } j=i-N  \text{ and } N+1\leq i\leq 2N-2; \\
-1+1/\zbar, &  \text{ if } j=i+1 \text{ and } N\leq i\leq 2N-3; \\
0, &\text{ otherwise.}
\end{array}\right.
\end{equation}

Note that $\zbar\bm{\widetilde{M}}_0=\bm{\widetilde{M}}$ in the neutral case, i.e. when $\psi_1\equiv 0$ and $\psi_2\equiv 0$. In order to define $\bm{\widetilde{M}}_1$ we need Hadamard's Lemma.

\begin{lemma}[Hadamard's Lemma~\citep{bruce}]
Let $f:U_{z_0}\subset\rr\longrightarrow \rr$ be a smooth function, where  $U_{z_0}$ is an open neighbourhood of $z_0$ and suppose $f^{(p)}(z_0)=0$ for all $p$ with $1\leq p \leq k$. Then there exist a smooth function $\overline{f}:U_{z_0}\longrightarrow \rr$ such that \[f(z)=f(z_0)+(z-z_0)^{k+1}\overline{f}(z)\] for all $z\in U_{z_0}$. When $k=0$ there are no such $p$ and the result also holds.
\end{lemma}

Applying Hadamard's Lemma to $a_x(z)-z$, for $k=1$ and $z_0=0$, we obtain 
\[
a_x(z)-z=z^2\int_0^1\int_0^1 s a''_x(zsu) dsdu.
\] 
Similarity, we rewrite all equations in (\ref{coefDB}) as 
\begin{equation} \label{coefDB2}
\begin{array}{l}
{\displaystyle b_x(z)=1-z+z^2\int_0^1\int_0^1sb''_x(zsu)dsdu,}\\\\
{\displaystyle c_x(z)=1-z+z^2\int_0^1\int_0^1sc''_x(zsu) dsdu, }\\\\
{\displaystyle d_x(z)=z+z^2\int_0^1\int_0^1sd''_x(zsu)dsdu.}
\end{array}
\end{equation}

Now we can define $\bm{\widetilde{M}}_1$ as follows
\[\left(\widetilde{M}_1\right)_{ij}=\left\lbrace \begin{array}{ll}
{\displaystyle\int_0^1\int_0^1sc''_{i\zbar}(zsu) dsdu\Big\rvert_{z=\overline{z}}}, & \text{ if } j= i-1 \text{ and } 2\leq i\leq N-1; \\\\

{\displaystyle\int_0^1\int_0^1sd''_{i\zbar}(zsu)dsdu\Big\rvert_{z=\overline{z}}}, & \text{ if } j= N+i \text{ and } 1\leq i\leq N-2; \\\\

{\displaystyle\int_0^1\int_0^1sa''_{(i-N)\zbar}(zsu) dsdu\Big\rvert_{z=\overline{z}}}, & \text{ if } j=i-N \text{ and } N< i\leq 2N-2; \\\\

{\displaystyle\int_0^1\int_0^1sb''_{(i-N)\zbar}(zsu) dsdu\Big\rvert_{z=\overline{z}}}, & \text{ if } j=i+1 \text{ and } N\leq i\leq 2N-3; \\\\

0, &\text{otherwise.}
\end{array}\right.\]

The next lemma shows there exist an upper bound for $\bm{\widetilde{M}}_1$ that does not depend on $\overline{z}$.

\begin{lemma}\label{cotaM1}
There exist a constant $C_1$ independent of $\overline{z}$ such that $\Vert\bm{\widetilde{M}}_1\Vert\leq C_1$.
\end{lemma}
\begin{proof}
Note that the function inside the integral of each term of matrix $\bm{\widetilde{M}}_1$ is a continuous function. Then, for variables $ u,s\in [0,1]$ and $\zbar<1$, there exist a maximum value that does not depend on $\zbar$. Therefore, each term of $\bm{\widetilde{M}}_1$ is bounded by a constant independent of $\zbar$ and so, the result follows.\end{proof}

We use the lemma bellow to find an upper bound for $\Vert\bm{\widetilde{M}}_0^{-1}\Vert$ in Proposition~\ref{normM0inv}.

\begin{lemma}[\cite{st}]\label{invnorm}
Let $\bm{K}\in\rr^{n\times n}$ be a matrix with non-positive off-diagonal elements, i.e. $K_{ij}\leq 0$ for $i\neq j$,  and suppose there exists a positive vector $\bm{r} > 0$ with $\bm{Kr} > 0$, then
\[\Vert \bm{K}^{-1}\Vert\leq\frac{\Vert \bm{r}\Vert}{\underset{i=1,...,n}{\min}(Kr)_i}.\]
\end{lemma}

\begin{proposition}\label{normM0inv}
Let $\bm{\widetilde{M}}_0$ be the matrix defined in (\ref{M0}). Then $\Vert\bm{\widetilde{M}}_0^{-1}\Vert\leq 1$.
\end{proposition}
\begin{proof}
First note that each off-diagonal element of the matrix $-\bm{\widetilde{M}}_0$ is non-positive. Now, consider the vector $\bm{r}$ such that $r_j=\overline{z}(N-1-(N-1-j)^2/(N-1)) $ for $j=1,...,N-1$ and $r_j=\overline{z}(N-1-(j-N)^2/(N-1)) $ for $j=N,...,2N-2$. Then  $(-\widetilde{M}_0r)_i=1$ for $i=1,...,2N-2$. Also, $\Vert \bm{r}\Vert=\overline{z}(N-1)$. Therefore, from Lemma~\ref{invnorm},  $\Vert\bm{\widetilde{M}}_0^{-1}\Vert\leq 1$. \end{proof}

We are now ready to find upper bounds for $\Vert\bm{\widetilde{M}}^{-1}\Vert$ and $\Vert \bm{\beta}+\bm{\widetilde{M}\widetilde{\overline{F}}}\Vert$ in the following propositions.

\begin{proposition}\label{cotaM}
There exists a constant $C$, independent of $\overline{z}$, for which
$$
\Vert\bm{\widetilde{M}}^{-1}\Vert\leq  C/\overline{z}.
$$
\end{proposition}
\begin{proof}
Note that $\bm{\widetilde{M}}^{-1}=\zbar^{-1}(\bm{I}+\zbar\bm{\widetilde{M}}_0^{-1}\bm{\widetilde{M}}_1)^{-1}\bm{\widetilde{M}}_0^{-1}$. Using Proposition~\ref{normM0inv} and Lemma~\ref{cotaM1}, and for a sufficiently small $\overline{z}$ (big $N$), we have $$\Vert \zbar\bm{\widetilde{M}}_0^{-1}\bm{\widetilde{M}}_1\Vert\leq \zbar\Vert\bm{\widetilde{M}}_1\Vert \leq C_2<1.$$ Therefore, from a property of \emph{Neumann series}, $$(I+\zbar\bm{\widetilde{M}}_0^{-1}\bm{\widetilde{M}}_1)^{-1}=\sum_{i=0}^\infty(-\zbar\bm{\widetilde{M}}_0^{-1}\bm{\widetilde{M}}_1)^i.$$ Thus, 
\begin{eqnarray*}
\Vert\bm{\widetilde{M}}^{-1}\Vert&\leq& \zbar^{-1}\Vert(\bm{\widetilde{M}}_0+\zbar\bm{\widetilde{M}}_1)^{-1} \Vert \\
&\leq& \zbar^{-1}\sum\limits_{i=0}^\infty\Vert(\zbar\bm{\widetilde{M}}_0^{-1}\bm{\widetilde{M}}_1)^i\Vert\\
& =&\zbar^{-1}\frac{1}{1-C_2}<\zbar^{-1}C.
\end{eqnarray*}
\end{proof}

\begin{proposition}\label{MF}
With the notation as in (\ref{erro}), there exists a constant $C_2$ that does not depend on $\overline{z}$, such that
\[\Vert \bm{\beta}+\bm{\widetilde{M}\widetilde{\overline{F}}}\Vert\leq \zbar^2C_2.\]
\end{proposition}

\begin{proof}
In order to simplify computations, we write $a_x(z)=z+z^2\gamma_{a_x}$, $b_x(z)=1-z+z^2\gamma_{b_x}$, $c_x(z)=1-z+z^2\gamma_{c_x}$ and $d_x(z)=z+z^2\gamma_{d_x}$, where $\gamma_{a_x},\gamma_{b_x},\gamma_{c_x}$ and $\gamma_{d_x}$ can be deduced from equations in (\ref{coefDB2}). Also we remind the reader that $f_1(x)=(x+1)/(2-\zbar)$, $f_2(x)=x/(2-\zbar)$, and $g_1(x)$ and $g_2(x)$ are defined in Subsection~\ref{odeDB}. For $2\leq i \leq N-2$ and $x=i\zbar$, $(\beta+\widetilde{M}\widetilde{\overline{F}})_i$ is equal to
\[
\begin{array}{l}
{\displaystyle c_x(\zbar)(f_2(x-\zbar)+\zbar g_2(x-\zbar))-(f_2(x)+\zbar g_2(x))+d_x(\zbar)(f_1(x)+\zbar g_1(x))}\\
{\displaystyle =(1-\zbar+\zbar^2\gamma_{c_x})\left(\frac{x-\zbar}{2-\zbar}+\zbar g_2(x-\zbar)\right)-\left(\frac{x}{2-\zbar}+\zbar g_2(x)\right)}\\
{\displaystyle \ \ \ +(\zbar+\zbar^2\gamma_{d_x})\left(\frac{x+1}{2-\zbar}+\zbar g_1(x)\right)}\\ 
{\displaystyle =\zbar(g_2(x-\zbar)-g_2(x))+\zbar^2\left(\frac{1}{2-\zbar}-g_2(x-\zbar)+\frac{x}{2-\zbar}\gamma_{c_x}+g_1(x)+\frac{x+1}{2-\zbar}\gamma_{d_x}\right)}\\
{\displaystyle \ \ \ +\zbar^3\left(g_2(x-\zbar)-\frac{\gamma_{c_x}}{2-\zbar}+\gamma_{d_x}g_1(x)\right)\leq \zbar(g_2(x-\zbar)-g_2(x))+\zbar^2\overline{C}_1}
\end{array} 
 \] 
 where $\overline{C}_1$ is a constant that is independent of $\zbar$. It is not hard to see that $g_2(x-\zbar)-g_2(x)=\zbar \overline{g}(x)$, where $\overline{g}(x)$ is a continuous function and so $g_2(x-\zbar)-g_2(x)\leq\zbar \overline{C}_2$, for $x\in [0,1]$.
Also,
\[
(\beta+\widetilde{M}\widetilde{\overline{F}})_1=-(f_2(\zbar)+\zbar g_2(\zbar))+d_{\zbar}(\zbar)(f_1(\zbar)+\zbar g_1(\zbar))\leq -\zbar g_2(z)+\zbar^2\overline{C}_3,
\] 
and we can deduce that $g_2(z)\leq \zbar \overline{C}_4$. Finally, $(\beta+\widetilde{M}\widetilde{\overline{F}})_{N-1}$ is equal to
\begin{multline*}
 c_{1-\zbar}(\zbar)(f_2(1-2\zbar)+\zbar g_2(1-2\zbar))-(f_2(1-\zbar)+\zbar g_2(1-\zbar))+d_{1-\zbar}(\zbar)\\
 \leq \zbar( \zbar+g_2(1-2\zbar)-g_2(1-z))+\zbar^2\overline{C}_5.
 \end{multline*} 

Therefore, $(\beta+\widetilde{M}\widetilde{\overline{F}})_i\leq\zbar^2\overline{C}_6$, for $1\leq i \leq N-1$. Similarly, for each $N\leq i\leq 2N-2$, we have $(\beta+\widetilde{M}\widetilde{\overline{F}})_i \leq\zbar^2\overline{C}$. 
Thus, the result follows.\end{proof}

\begin{proof}[Proof of Theorem~\ref{rDB}]
Using Propositions~\ref{cotaM} and ~\ref{MF} in (\ref{erro}), and recalling that $\overline{F}_1=0$ and $\overline{F}_{2N}=1$, we conclude that the approximation of the fixation probability vector given by $\bm{\overline{F}}$ is of order $\zbar$ and so Theorem~\ref{rDB} holds.
\end{proof}

\section{Proof of Proposition~\ref{fixation-Complete}}\label{AppendixB}

Let \( Q_r^{(N)} = \frac{q_r}{p_r} \) and \( P_j^{(N)} = P_j = \prod_{r=1}^j Q_r \) for \( j = 1, \dots, N \). We have
\[
Q_r = 1 + \frac{1}{N} \chi\left(\frac{r}{N}\right) + O(N^{-2}).
\]

Then,
\begin{align}\label{AppB-1}
\log P_j &= \sum_{r=1}^j \log Q_r \nonumber\\
&= \sum_{r=1}^j \log \left(1 + \frac{1}{N} \chi\left(\frac{r}{N}\right) + O(N^{-2}) \right) \nonumber \\
&= \sum_{r=1}^j \left( \frac{1}{N} \chi\left(\frac{r}{N}\right) + O(N^{-2}) \right) \nonumber \\
&= \frac{1}{N} \sum_{r=1}^j \chi\left(\frac{r}{N}\right) + O(N^{-1}). 
\end{align}

Hence,
\begin{align*}
P_j &= \exp\left( \frac{1}{N} \sum_{r=1}^j \chi\left(\frac{r}{N}\right) \right) \left(1 + O(N^{-1}) \right) \\
&= \exp\left( \frac{1}{N} \sum_{r=1}^j \chi\left(\frac{r}{N}\right) \right) + O(N^{-1}).
\end{align*}

Therefore,
\[
J_{i,N} = \frac{1}{N} \sum_{j=1}^{i-1} \exp\left( \frac{1}{N} \sum_{r=1}^j \chi\left(\frac{r}{N}\right) \right) + O(N^{-1}).
\]

From (\ref{AppB-1}), 
\begin{align*}
\log P_j - \int_0^{\frac{j}{N}} \chi(s)\, ds 
=& \left( \log P_j - \frac{1}{N} \sum_{r=1}^j \chi\left(\frac{r}{N}\right) \right)\\
&+ \left( \frac{1}{N} \sum_{r=1}^j \chi\left(\frac{r}{N}\right) - \int_0^{\frac{j}{N}} \chi(s)\, ds \right)= O(N^{-1}).
\end{align*}

Thus, recalling the definition of \(\xi\) before Proposition~\ref{fixation-Complete},
\begin{equation}\label{AppB-2}
    P_j = \xi\left( \frac{j}{N} \right) \exp\left(O(N^{-1})\right) 
= \xi\left( \frac{j}{N} \right)\left(1 + O(N^{-1})\right).
\end{equation}
Using the bound
\begin{equation}\label{AppB-3}
  \Bigg\lvert \frac{1}{N} \sum_{j=1}^{i-1} \xi\left(\frac{j}{N}\right) - \mathcal{J}\left(\frac{i}{N}\right) \Bigg\rvert
\leq \frac{\|\xi'\|_\infty}{2N} = O(N^{-1}),
\end{equation}
where $\|\xi'\|_\infty = \sup_{s \in [0,1]} \lvert \xi'(s)\rvert$ is the sup-norm of the derivative of $\xi$. Also,
\begin{equation}\label{AppB-4}
    \Bigg\lvert J_{i,N} - \frac{1}{N} \sum_{j=1}^{i-1} \xi\left( \frac{j}{N} \right) \Bigg\rvert
\leq \frac{1}{N} \sum_{j=1}^{i-1} \lvert P_j - \xi\left( \frac{j}{N} \right) \rvert = O(N^{-1}).
\end{equation}
It follows from (\ref{AppB-3}) and (\ref{AppB-4}) that
\[
J_{i,N} = \mathcal{J}\left( \frac{i}{N} \right) + O(N^{-1}), \quad
J_{N,N} = \mathcal{J}(1) + O(N^{-1}).
\]

But we have, for \(i = 1, \dots, N\),
\begin{align*}
\tilde h_N(i) - \frac{\mathcal{J}(\frac{i}{N})}{\mathcal{J}(1)} 
&= \frac{1 + N \mathcal{J}(\frac{i}{N}) + O(1)}{1 + N \mathcal{J}(1) + O(1)} 
- \frac{\mathcal{J}(\frac{i}{N})}{\mathcal{J}(1)} \\
&= \frac{\mathcal{J}(1) - \mathcal{J}(\frac{i}{N}) + O(1)}{N (\mathcal{J}(1))^2 + O(1)} 
= O(N^{-1}). \qed
\end{align*}\\

\end{appendices}

\bibliographystyle{abbrvnat}
\bibliography{bibliography}

\end{document}